\documentclass[12pt]{amsart}
\usepackage{txfonts}
\usepackage{soul}
\usepackage{amscd,amssymb,mathabx,mathtools}
\usepackage[arrow,matrix,graph,frame,poly,arc,tips]{xy}
\usepackage[dvipsnames]{xcolor}
\usepackage{graphicx}
\usepackage{calrsfs}
\usepackage[labelfont=rm]{subcaption}
\usepackage{tikz}
\usetikzlibrary{decorations.markings}
\usetikzlibrary{shapes}
\usetikzlibrary{backgrounds}
\usetikzlibrary{calc}
\usetikzlibrary{shapes.misc, positioning}
\usepackage{mdwlist}
\usepackage{xfrac}
\DeclareCaptionSubType*{figure}

\usepackage{multirow}
\usepackage{enumerate}
\usepackage{verbatim}
\usepackage{comment}
\usepackage{todonotes}

\DeclarePairedDelimiter{\floor}{\lfloor}{\rfloor}

\DeclarePairedDelimiter{\abs}{\lvert}{\rvert}

\newcommand\on[1]{\operatorname{#1}}
\newcommand{\lst}[0]{\on{St}^{\lim}}
\newcommand{\pa}[0]{\on{par}}

    \newcommand\ba{\begin{align*}}
    \newcommand\ea{\end{align*}}
    \newcommand\be{\begin{enumerate}}
    \newcommand\ee{\end{enumerate}}
    \newcommand\bpp{\begin{prop}}
    \newcommand\epp{\end{prop}}
    \newcommand\bpb{\begin{prob}}
    \newcommand\epb{\end{prob}}
    \newcommand\bd{\begin{defn}}
    \newcommand\ed{\end{defn}}
    \newcommand\bh{\begin{hint}}
    \newcommand\eh{\end{hint}}
  \newcommand{\sq}[0]{\sqrt{2}}

  \newcommand{\pr}[0]{\mathrm{pr}}
  \newcommand{\conv}[0]{\mathrm{conv}}

    \newcommand\bN{\mathbb{N}}
    
    \newcommand\bR{\mathbb{R}}
    \newcommand\bQ{\mathbb{Q}}
    \newcommand\bZ{\mathbb{Z}}
    \newcommand\Z{\mathbb{Z}}
    \newcommand\bH{\mathbb{H}}

    \newcommand\EE{\mathcal{E}}
    \newcommand\FF{\mathcal{F}}

    \newcommand\II{\mathcal{I}}
    
    \newcommand\LL{\mathcal{L}}
    \newcommand\PP{\mathcal{P}}
\newcommand\UU{\mathcal{U}}
\newcommand\VV{\mathcal{V}}
\newcommand\WW{\mathcal{W}}
\newcommand\XX{\mathcal{X}}

\newcommand\vol{\operatorname{vol}}

    \newcommand\supp{\operatorname{supp}}
    \newcommand\Id{\operatorname{Id}}

    \newcommand\diam{\operatorname{diam}}

    \newcommand\cc{\mathrm{cc}}
    
    \newcommand\union{\mathrm{union}}

    \DeclareMathOperator\Homeo{Homeo}

    \newcommand\cM{\mathcal{M}}

    \newcommand\sse{\subseteq}
    \newcommand\co{\colon}
    \newcommand\cH{ G}    
    \DeclareMathOperator\im{im}
    
    \DeclareMathOperator\fix{fix}
    
    \DeclareMathOperator\Aut{Aut}

    \DeclareMathOperator\Diff{Diff}

    \DeclareMathOperator\stab{stab}
    
    \DeclareMathOperator\ro{RO}
    
    \DeclareMathOperator\inte{int}
    \DeclareMathOperator\cl{cl}
    \DeclareMathOperator\fr{fr}
    \DeclareMathOperator\ext{ext}
    \DeclareMathOperator\agape{AGAPE}
    \DeclareMathOperator\cont{Cont}
    \DeclareMathOperator\appl{appl}
    \DeclareMathOperator\Th{{Th}}
    \DeclareMathOperator\act{{Act}}

    \DeclareMathOperator\embed{{Embed}}
    \DeclareMathOperator\param{{Param}}
     \DeclareMathOperator\collar{{collar}}

\newcommand\suppe{\operatorname{{supp}^{\mathrm{e}}}}
\newcommand\bG{\mathbf G}
\newcommand\bM{\mathbf M}

    \def\thetitle{{First order rigidity of homeomorphism groups of manifolds}}
    \def\theauthors{{J. de la Nuez Gonz\'alez, Sang-hyun Kim, Thomas Koberda}}
    \usepackage{hyperref}
    \hypersetup{
    	colorlinks=false,
    	plainpages,
    	urlcolor=black,
    	linkcolor=black
    	pdftitle=   \thetitle,
    	pdfauthor=  {\theauthors}
    }

    \theoremstyle{plain}
    \newtheorem{thm}{Theorem}[section]
    \newtheorem{lem}[thm]{Lemma}
    
    \newtheorem{cor}[thm]{Corollary}
    \newtheorem{prop}[thm]{Proposition}
    
    \newtheorem{que}[thm]{Question}
    
    \newtheorem*{claim*}{Claim}

    \theoremstyle{remark}
    
    \newtheorem{rem}[thm]{Remark}

    \theoremstyle{definition}
    \newtheorem{defn}[thm]{Definition}
    \newtheorem{prob}{Problem}[section]

\begin{document}
    \title\thetitle
    \date{\today}
    
    \author[S. Kim]{Sang-hyun Kim}
    \address{School of Mathematics, Korea Institute for Advanced Study (KIAS), Seoul, 02455, Korea}
    \email{skim.math@gmail.com}
    \urladdr{https://www.kimsh.kr}

    \author[T. Koberda]{Thomas Koberda}
    \address{Department of Mathematics, University of Virginia, Charlottesville, VA 22904-4137, USA}
    \email{thomas.koberda@gmail.com}
    \urladdr{https://sites.google.com/view/koberdat}
    
    \author[J. de la Nuez Gonz\'alez]{J. de la Nuez Gonz\'alez}
    \address{School of Mathematics, Korea Institute for Advanced Study (KIAS), Seoul, 02455, Korea}
    \email{jnuezgonzalez@gmail.com}

    \begin{abstract}
    	For every compact, connected manifold $M$, we prove the existence of a sentence $\phi_M$ in the language of groups such that
the homeomorphism group of another compact manifold $N$ satisfies
 $\phi_M$ if and only if $N$ is homeomorphic to $M$. 
	We prove an analogous statement for groups
    	of homeomorphisms preserving Oxtoby--Ulam probability measures.
    \end{abstract}
    
    \maketitle
    
\setcounter{tocdepth}{1}
\tableofcontents
    

\section{Introduction}\label{sec:intro}
This article relates topological manifolds, homeomorphism groups, and first order theories. For us, a \emph{manifold} will mean a second countable, metrizable topological space, each point of which has a closed neighborhood homeomorphic to a fixed closed Euclidean ball. In particular, a manifold is allowed to have boundary.
The \emph{first order theory} (or \emph{elementary theory}) of a group is the collection of the first order sentences (i.e. sentences that do not involve quantification of subsets) which are valid in the group; see Section~\ref{ss:model theory} for details.

We begin by introducing the main objects of study.
For a manifold $M$ (possibly equipped with a probability
measure $\mu$), we let $\Homeo(M)$ and $\Homeo_\mu(M)$ denote the homeomorphism group of $M$ and its $\mu$--preserving subgroup, respectively. We denote by $\Homeo_0(M)$ and $\Homeo_{0,\mu}(M)$ the identity components of $\Homeo(M)$ and $\Homeo_\mu(M)$, respectively. 
For general topological spaces $X$ and $Y$, we write $X\cong Y$ if $X$ and $Y$ are
homeomorphic.

We denote by $\cM$
 the class of all pairs $(M,\cH)$,
where $M$ is a compact, connected manifold
and $\cH$ is a group satisfying
\[
\Homeo_0(M)\le \cH\le\Homeo(M).
\]
We also let $\cM_{\vol}$ denote
the class of all $(M,\cH)$
where $M$ is further assumed to be
equipped with some \emph{Oxtoby--Ulam measure} $\mu$ (that is, a nonatomic Borel probability measure having full support and assigning measure zero to the boundary),
and $\cH$ is a group satisfying
\[
\Homeo_{0,\mu}(M)\le \cH\le\Homeo_\mu(M).
\]
Note that in this case, we have
\[
\Homeo_{0,\mu}(M)=\Homeo_0(M)\cap\Homeo_\mu(M);
\]
cf.~\cite{Fathi1980}.

\begin{rem}\label{rem:vol}
In statements that apply to both of the classes $\cM$ and $\cM_{\vol}$, we will often use the notation $\cM_{(\vol)}$; 
in such a statement, the choices of formulae may differ, even when the formulae share the same names.
\end{rem}
We will later modify the definitions of the classes $\cM_{(\vol)}$ slightly so that only manifolds of dimension at least two are considered; see the remark at the end of Section~\ref{ss:top}.

To motivate the discussion in this article, we consider the general
\emph{reconstruction} problem of an object from its group of automorphisms.
For a general object $X$ in some category, it is natural to ask the degree to
which the group of automorphisms $\Aut(X)$ determines the object $X$. This question
is not completely precise, since the terms ``degree" and ``determine" do not have
a mathematical meaning here. In our context, the object $X$ will
always be a compact manifold, possibly with boundary, and the group of automorphisms
will be one of the groups of homeomorphisms we have defined already.

The precise meaning of ``degree" will be ``the information encoded in the first order
theory", and ``determine" will precisely mean ``reconstruct the homeomorphism type".
That is,
the goal of this paper is to investigate,
under the assumption that $(M,\cH)\in\cM_{(\vol)}$,
the extent to which the first order theory of $\cH$ can be used to reconstruct
the homeomorphism type of $M$.

Of course, the first order theory of the homeomorphism group of a manifold is not the
only data one can investigate for the reconstruction of the homeomorphism type of the
underlying manifold. Perhaps the most basic invariant of the group of homeomorphisms
of a manifold $M$ is its isomorphism type.

It is a classical result of Whittaker that the isomorphism type of
the homeomorphism group of a compact manifold determines the homeomorphism type of the
    underlying manifold in the following sense:
    
\begin{thm}[See~\cite{whittaker}]\label{thm:Whittaker}
Let $M$ and $N$ be compact manifolds, and suppose \[\phi\colon \Homeo(M)\longrightarrow\Homeo(N)\] is an isomorphism of groups.
Then there exists a homeomorphism \[\psi\colon M\longrightarrow N\] such that for all $f\in\Homeo(M)$, we have $\phi(f)=\psi\circ f\circ\psi^{-1}$.
    \end{thm}
    
Whittaker's result has been generalized by a number of authors; see Chapter 3 of~\cite{KK2021book} for a survey. For instance, combining the work of Bochner--Mongomery~\cite{BM1945} on Hilbert's fifth problem and of Takens on smooth conjugation between diffeomorphisms~\cite{Takens1979} (cf.~\cite{filipkiewicz}), one obtains that if $M$ and $N$ are smooth and closed, and if the diffeomorphism groups $\Diff^k(M)$ and $\Diff^{\ell}(N)$
are isomorphic as groups, then $k=\ell$ and 
each isomorphism
between the groups is induced by 
some $C^k$--diffeomorphism between $M$ and $N$.

In the continuous category, a different generalization was given by Rubin.
We say that a topological action of a group $G$ on a topological space $X$ is \emph{locally dense} if for each pair $(x,U)$ of a point $x\in X$
and a neighborhood $U\sse X$ of $x$,
the orbit $Z$ of $x$ by the action of the group
\[
G[U]:=\{g\in G\mid g(y)=y\text{ for all }y\not\in U\}
\]
is somewhere dense; that is, the closure of $Z$ has nonempty interior.
Rubin's Theorem can be stated as followsL

\begin{thm}[\cite{Rubin1989}]\label{thm:rubin-intro}
Let $X_1$ and $X_2$ be perfect, locally compact, Hausdorff topological spaces, and let $G_i\leq \Homeo(X_i)$ be locally dense subgroups for $i\in\{1,2\}$. 
If there exists an isomorphism if groups \[\phi\colon G_1\longrightarrow G_2,\] then there exists a homeomorphism \[\psi\colon X_1\longrightarrow X_2\] such that for all $g\in G_1$, we have $\phi(g)=\psi\circ g\circ\psi^{-1}$. \end{thm}

The reason for considering the (\emph{a priori} much weaker)
first order theory of a homeomorphism group instead
of the full isomorphism type of the homeomorphism group is because 
an isomorphism between two groups of homeomorphisms is a rather unwieldy
(and frankly unnatural) piece of data. Homeomorphism groups of manifolds are generally much too large to write down, and directly accessing homomorphisms between them is 
practically impossible. Therefore, we will be interested in more finitary ways of investigating homeomorphism groups of manifolds, namely through their elementary theories.

With this goal in mind,
we consider the language of groups, which consists of a binary operation (interpreted as the group operation) and a constant (interpreted as the identity element). Models of the theory of groups are just sets with interpretations of the group operation and identity element which satisfy the axioms of groups. We say that two groups $G_1$ and $G_2$ are \emph{elementarily equivalent}, written $G_1\equiv G_2$, if a first order sentence in the language of groups holds in $G_1$ if and only if it holds in $G_2$; this is sometimes expressed as saying that the theories of $G_1$ and $G_2$ agree, i.e. \[\Th(G_1)=\Th(G_2).\] Here, \emph{first order} refers to the scope of quantification, which is allowed to range over elements (as opposed to subsets, relations, or functions).

Philosophically, the
reason for considering first order theories as opposed to second (or higher) order
theories is that, whereas it is typically not controversial what ``elements" in a
structure refer to, the objects which are admitted as ``subsets" of a structure depend
on the underlying choice of set theory; there is generally no agreement on acceptable
axioms for set theory. A further ``constructive" benefit of the first order theory of a structure is that it is a \emph{syntactic}
invariant, in the sense that it records a list of ``true statements" about the structure
which can, in principle, be recorded.

\emph{First order rigidity} in a class of structures refers to the phenomenon
where two elementarily equivalent structures are automatically isomorphic. Of course,
a class of structures may or may not enjoy first order rigidity, and \emph{a priori}
elementary equivalence is a much coarser equivalence relation than isomorphism. Because
of general model-theoretic phenomena such as the upward L\"owenheim--Skolem Theorem
(which says roughly that once one has an infinite model of a theory then
one has elementarily equivalent models of arbitrarily high cardinality), one
should restrict one's attention to models of the same cardinality; even so,
for countable groups, it is not the case that elementary equivalence implies isomorphism.
A typical example is the class of nonabelian free groups, wherein any two such groups
are elementarily equivalent~\cite{KM2006,Sela10}.

The content of this paper
fits within a tradition of results establishing that certain classes of
structures do enjoy first order rigidity, such as lattices in higher rank~\cite{avni-etal}, function fields~\cite{duret1,duret2,vidaux,Pop02}, rings~\cite{leloup,greenfeld}, finite--by--abelian groups~\cite{oger}, and linear groups~\cite{plotkin}, cf.~\cite{tent-segal23}.
Moreover, the themes of this paper are compatible with the
philosophy that one should like to distinguish between objects that are difficult to access directly via finite syntactic proxies.

\subsection{Elementary equivalence implies homeomorphism}
    
Our main result says precisely that 
two compact, connected manifolds have elementarily equivalent homeomorphism groups
if and only if the underlying manifolds are homeomorphic to each other.
More strongly, for each compact connected manifold $M$ 
 we prove the existence of a group theoretic sentence that asserts ``I am homeomorphic to $M$'':

\begin{thm}\label{thm:main}
For each compact, connected manifold $M$,
there exists a sentence $\phi_M^{(\vol)}$ in the language of groups such that 
when $(N,H)\in\cM_{(\vol)}$, 
we have that  
\[
\phi_M^{(\vol)}\in\Th(H)\qquad\text{if and only if}\qquad N\cong M.\]
\end{thm}

In other words, the theories of homeomorphism groups of manifolds
are \emph{quasi-finitely axiomatizable} within the class
$\cM_{(\vol)}$, a property that is stronger than first order rigidity.

In Theorem~\ref{thm:main}, we emphasize that $M$ and $N$ are not assumed to have any further structure, such as a smooth or piecewise-linear structure. 
We thus generalize Whittaker's result without relying on it, and produce for each manifold a finite, group--theoretic sentence that certifies homeomorphism or non--homeomorphism with the manifold. 
The sentences $\phi_M$ and $\phi_M^{\vol}$ are produced explicitly insofar as is possible, though in practice it would be a rather tedious task to record them.
We also note that the connectedness hypothesis for $N$ can also be dropped from the theorem, thus justifying the claim in the abstract;
see Corollary~\ref{cor:basic-form}, for instance.

A further motivation for Theorem~\ref{thm:main} that does not arise from philosophical
or foundational considerations
centers around the following dynamical question; a number of other related questions are enumerated in Section~\ref{sec:questions}.
    
    \begin{que}\label{que:homeo-action}
    Let $M$ be a compact, connected manifold. Under what conditions is there a finitely generated (or countable) group $G_M\leq \Homeo(M)$ such that whenever $N$ is a compact manifold with $\dim M=\dim N$ on which $G_M$ acts faithfully with a dense orbit, we have $M\cong N$?
    \end{que}
    
    Related results for actions of the full homeomorphism group of $M$ are given by Chen--Mann~\cite{chen-mann}. They show that
    if the identity component of $\Homeo(M)$ acts transitively on a connected manifold or CW--complex $N$, then $N$ is homeomorphic
    to a cover of a configuration space of points of $M$. In our context, we have the following immediate consequence of the downward
    L\"owenheim--Skolem Theorem:
    
    \begin{cor}\label{cor:L-S}
To each compact connected manifold $M$ one can associate a countable group $G_M\leq \Homeo(M)$ which is elementarily equivalent to $\Homeo(M)$, such that
for two compact, connected manifolds $M$ and $N$ we have
\[G_M\equiv G_N\quad\textrm{if and only if}\quad M\cong N.\] 
In particular $G_M\cong G_N$ if and only if $M\cong N$.
    \end{cor}

    \begin{rem}
    	In~\cite{Rubin1989}, there is a cryptic announcement of a version of Theorem~\ref{thm:main}. In particular, Rubin claims that under
    	the assumption $V=L$ (i.e.~G\"odel constructibility) that two arbitrary connected manifolds are homeomorphic if and only if their
    	homeomorphism groups are elementarily equivalent; 
     it is likely that
     he implicitly made a few other assumptions (e.g.~excluding manifolds
     with boundary) to avoid trivial counterexamples such as $\Homeo(0,1)\cong\Homeo[0,1]$. To the knowledge of the authors, the paper
    	bearing the title announced in~\cite{Rubin1989} never appeared, and neither did any result
    	(of any authors whatsoever) proving first order rigidity of homeomorphism
    	groups of manifolds; cf.~a related MathOverflow post~\cite{MO-Rubin}.
	We note that we only establish results for compact manifolds, in contrast to Rubin's original announcement.
    
    Rubin's original reason for considering the assumption $V=L$ remains unclear,
    and perhaps the goal was to
     promote first order equivalence to second order
     equivalence, using the assumption $V=L$ to conclude the resulting second
     order equivalent structures are isomorphic;~cf.~\cite{ajtai79}. In work that is
     ongoing at the time of this writing, the second and third author, together with
     J.~Hanson and C.~Rosendal have established that first order rigidity for homeomorphism
     groups of noncompact manifolds depends on the choice of set theory used.
    \end{rem}

Our proof of Theorem~\ref{thm:main} largely consists of two parts. 
The first part is constructing an expansion of the language of group theory to a seemingly more powerful language, called $L_{\agape}$. 
The universe of an $L_{\agape}$ structure corresponding to $(M,G)\in\cM_{(\vol)}$ will contain the group $G$, the regular open sets $\ro(M)$ of $M$, the real numbers $\bR$, the set of continuous maps $C^0(\bR^k,\bR^\ell)$ for \[k,\ell\in\omega=\{0,1,2,\ldots\}\]
and the discrete subsets of $M$.
Since the expansion is specified by first order definitions, we deduce the following, which roughly means that each sentence in the theory of $\agape(M,\cH)$ can be interpreted (in a way that is uniform in $(M,\cH)$) as a sentence in the theory of the group $\cH$; see Section~\ref{sec:prelim} for a precise definition of \emph{uniform interpretation}:

\begin{thm}\label{thm:interpret}
For $(M,G)\in \cM_{(\vol)}$,
the structure $\agape(M,G)$ is uniformly interpretable in the group structure $\cH$.
\end{thm}

The second part of the proof consists in showing that the $\agape$ language can express the sentence that ``I am homeomorphic to $M$'':

\begin{thm}\label{thm:iamM}
For each $(M,G)\in\cM_{(\vol)}$,
there exists an $L_{\agape}$--sentence $\psi^{(\vol)}_{M,G}$ 
such that for all $(N,H)\in\cM_{(\vol)}$,
we have  
\[\psi^{(\vol)}_{M,G}\in\Th\agape(N,H)\quad
\text{if and only if}\quad N\cong M.\]
\end{thm}

By Theorem~\ref{thm:interpret},
we can interpret $L_{\agape}$--sentences \[\psi_{M,\Homeo(M)}\quad \textrm{and} 
\quad \psi_{M,\Homeo_\mu(M)}\] as group theoretic sentences $\phi_M$ and $\phi_M^{\vol}$ respectively, which distinguish $M$ from all the other non-homeomorphic manifolds $N$;
see Lemma~\ref{lem:inter-models} for a more formal explanation. We thus obtain a proof of Theorem~\ref{thm:main}.
    
\begin{rem}\label{rem:other-results}
    A few of the first order rigidity results obtained in this paper can be obtained for a substantially larger class of groups of homeomorphisms which are much
    smaller than the full group of homeomorphisms of $M$; see the recent paper~\cite{KdlN2024}. For certain groups
    of homeomorphisms that are ``sufficiently dense" in the full group of homeomorphisms
    (called \emph{locally approximating groups}),
    one can prove that the first order theory of these groups determines the underlying
    manifold up to homotopy equivalence. The first order theory of 
    locally approximating groups of homeomorphisms is substantially weaker than the
    theory of the full homeomorphism group; indeed, in~\cite{KdlN2024} we
    can only recover the main theorem of this paper for closed, triangulated
    manifolds for which the Borel Conjecture holds (i.e.~homotopy equivalence implies homeomorphism, which is false in general); among the main technical difficulties of this paper are constructing methods which work for all manifolds (including ones admitting no triangulation), and including manifolds with boundary. We are able to prove the main result because of the remarkable expressivity of the first order theory of the full group of homeomorphisms of a manifold.
    
    We note finally that the first order theory of the full homeomorphism group of a manifold $M$
    is expressive enough to interpret the full second order theory of countable
    subsets of $\Homeo(M)$, something which is not
    possible in a general locally approximating group of homeomorphisms; see the recent paper~\cite{KdlN2023}.
\end{rem}

  \subsection{Outline of the paper}   
    The paper is devoted to proving Theorem~\ref{thm:main} in several steps, each of which builds on the previous.
    Section~\ref{sec:prelim} gathers basic results from geometric topology and model theory, and fixes notation. In Section~\ref{sec:sorts}, we introduce the language and the structure of primary interest for us, called $\agape$. 
    In this structure, we interpret the regular open sets in $G$, and construct formulae that encode various topological properties of regular open sets.
    Section~\ref{sec:arith} interprets second order arithmetic using regular open sets
    and actions of homeomorphisms on them. Section~\ref{sec:points} encodes individual points of a manifold, together with the exponentiation map. Section~\ref{sec:balls} interprets the dimension
    of a manifold, as well as certain definably parametrized embedded Euclidean balls.
    Section~\ref{sec:collar} definably parametrizes collar neighborhoods of the boundary of a compact manifold. Section~\ref{sec:homeo} proves Theorem~\ref{thm:main} by interpreting a result of Cheeger--Kister~\cite{CK1970} and by encoding embeddings of manifolds into Euclidean spaces that are
    ``sufficiently near'' to a fixed embedding. 
    We conclude with some questions in Section~\ref{sec:questions}.
  
\section{Preliminaries}\label{sec:prelim}
In this section, we gather some notation, background, and generalities.

\subsection{Transitivity of balls in manifolds}
The high degree of transitivity of the action of homeomorphism groups
on balls in manifolds is crucial for this paper. We begin with the
following fundamental fact about Oxtoby--Ulam measures.

\begin{thm}[von Neumann~\cite{vonNeumann1961},
Oxtoby--Ulam~\cite{oxtoby-ulam}]\label{thm:ou}
If $\mu$ and $\nu$ are Oxtoby--Ulam measures on a compact connected manifold $M$, then there exists a homeomorphism $h$ of $M$ isotopic to the identity and 
fixing $\partial M$ such that $h_*\mu=\nu$.
\end{thm}
    
Thus, for Oxtoby--Ulam measures, the groups of measure-preserving homeomorphisms of $M$ are all conjugate to each other. In particular, each $(M,G)\in\cM_{\vol}$ corresponds to a measure that is unique up to topological conjugacy.
We will therefore refer to groups of \emph{measure-preserving homeomorphisms} without specifying a particular Oxtoby--Ulam measure.  
We refer the reader to~\cite{Fathi1980,Banyaga1997} for generalities about measure-preserving homeomorphisms.

A group theoretic interpretation of (certain) balls in a manifold will be another crucial step in this paper. The importance of being able to identify regular open sets which are homeomorphic to balls comes from the following lemma, which is originally due to Brown~\cite{Brown1962b} for the first two parts, and to Fathi~\cite{Fathi1980} for the remainder. 
One may view this as a natural generalization of the fact that every compact connected 2-manifold can be obtained by gluing up the boundary of a polygon in a suitable way.
See also~\cite{CV1977,LeRoux2014} for more details.
In the statement of the lemma, $B^n(r)$ means the compact ball of radius $r>0$ with the center at the origin in $\bR^n$.

\begin{lem}[\cite{Brown1962b,Fathi1980}; cf.~{\cite[Chapter 17]{CV1977}}]\label{lem:brown-measure}
For each compact connected $n$--manifold $M$, there exists a continuous surjection \[f\colon B^n(1)
\longrightarrow M\] such that the following hold:
\begin{enumerate}[(i)]
\item
the restriction of $f$ on  $\inte B^n(1)$ is an embedding onto an open dense subset of $M$;
\item
we have $f(\inte B^n(1))\cap f(\partial B^n(1))=\varnothing$,
and in particular, $\partial M\sse  f(\partial B^n(1))$;
\end{enumerate}
If
 $M$ is equipped with an Oxtoby--Ulam measure $\mu$, we can further require the following:
\begin{enumerate}[(i)]
\setcounter{enumi}{2}
\item
we have $\mu(f(\partial B^n(1)))=0$;
\item\label{p:brown-leb}
the measure $\mu$ is the pushforward of Lebesgue measure by $f$.
\end{enumerate}
\end{lem}
The conditions (i) and (ii) already imply that an Oxtoby--Ulam measure on $M$ exists.
For instance, one can pull back the Lebesgue measure on a ball using the surjection \[B^n(1)\longrightarrow M\] from Lemma~\ref{lem:brown-measure}. 
The condition (iv) is also easy to be obtained from the previous conditions and Theorem~\ref{thm:ou}; see also~\cite{GP1975}. 

For a possbily non-compact manifold, we have the following variation also due to Fathi, which loosens the condition on the surjectivity of the map.
\begin{lem}[\cite{Fathi1980}]\label{lem:brown-measure2}
If a connected $n$--manifold $M$ has nonempty boundary
and if $M$ is equipped with a nonatomic, fully supported Radon measure $\mu$ that assigns zero measure to $\partial M$, then there exists an open embedding 
\[
f\co \bH^n_+:=\{(x_1,\ldots,x_n)\in\bR^n\mid x_n\ge0\}\longrightarrow M
\]
such that the following hold:
\be[(i)]
\item $f(\inte \bH^n_+)\sse\inte M$ and $f(\partial\bH^n_+)\sse\partial M$;
\item $M\setminus f(\bH^n_+)$ is closed and of measure zero.
\ee
\end{lem}

We call a topologically embedded image of $B^n(1)$ in a manifold $M^n$ a \emph{ball}. 
The same goes for an \emph{open ball} in $M$.
If there exists an embedding \[h\co B^n(2)\longrightarrow M,\] then the image $h(B^n(1))$ is called a \emph{collared ball}~\cite[Chapter 17]{CV1977}. The same goes for a \emph{collared open ball}.
In the case when $M$ is equipped with an Oxtoby--Ulam measure $\mu$, we say a collared ball $B$ is $\mu$--\emph{good} (or, simply \emph{good}) if $\partial B$ has measure zero.
There exists an arbitrarily small covering of $M$ by $\mu$--good balls~\cite{Fathi1980}.
For brevity of exposition, 
by a \emph{good ball}, we mean both a collared ball in the context of $(M,G)\in\cM$ and a \emph{$\mu$--good ball} in the context of $(M,G)\in\cM_{\vol}$. The same goes for a good open ball.
Note that a good ball is always contained in the interior of $M$.

Recall the topological action of a group $G$ on $\inte M$ is \emph{path--transitive} 
if for all paths \[\gamma\co I\longrightarrow \inte M\] and for all neighborhoods $U$ of $\gamma(I)$
there exists $h\in G[U]$ such that $h(\gamma(0))=\gamma(1)$. 
We say the action of $G$ on $\inte M$ is \emph{$k$--transitive} if it induces a transitive action on the configuration space of $k$ distinct points in $\inte M$.
A path-transitive action on $\inte M$ is always $k$--transitive whenever $\dim M>1$; see~\cite[Lemma 7.4.1]{Banyaga1997}. 
Let us note the following fundamental facts on various notions of transitivity in manifolds. 

\begin{lem}\cite[Corollaries 2.1 and 2.2]{LeRoux2014}\label{lem:transitive}
For $(M,G)\in\cM_{(\vol)}$ with $\dim M>1$, we have the following.
\be
\item\label{p:trans-path}
The action of $G$ on $\inte M$ is path--transitive and $k$--transitive for all $k>0$.
\item\label{p:transitive-measure}
If $B_1$ and $B_2$ are good balls of the same measure in an open connected set $U\sse M$,
then there exists $g\in G[U]$ such that $g(B_1)=B_2$.
\ee
\end{lem}
\begin{proof}
The path--transitivity of part (\ref{p:trans-path}) is well-known; see~\cite[Section 7.7]{Banyaga1997} for $G=\Homeo_0(M)$, and~\cite[p.~85]{Fathi1980} for $G=\Homeo_{0,\mu}(M)$. The $k$--transitivity follows immediately.

The case when $U=M$ in part (\ref{p:transitive-measure}) is precisely given in~\cite[Corollary 2.2]{LeRoux2014} by Le Roux, 
based on the Annulus Theorem of Kirby~\cite{Kirby1969} and Quinn~\cite{Quinn1982} as well as the Oxtoby--Ulam theorem. 
In general, we can exhaust the topological manifold $U$ by a sequence of compact bounded manifolds $\{M_i\}$ so that some $M_i$ contains $B_1$ and $B_2$ in its interior; this can be seen from~\cite{Quinn1988}, as explained in~\cite{MO-Quinn}. We can further require that $\partial M_i$ has measure zero by countable additivity. Applying Le Roux's argument for $M_i$, we obtain the desired transitivity.\end{proof}

\begin{lem}\label{lem:measure-ball}
Let $M$ be a compact, connected $n$--manifold with $n\ge2$,
equipped with an Oxtoby--Ulam measure $\mu$.
If $U\sse M$ is an open connected subset, then for each positive real number $r<\mu(U)$, there exists a good ball of measure precisely $r$ inside $U$.
Moreover, we may require that $U\setminus B$ is connected.
\end{lem}
\begin{proof}
Note the general fact that for a connected open subset $U$ of $M$ and for a collared ball $B$ in $U\cap\inte M$, the set $U\setminus B$ is connected; this can be seen from the fact that a collared ball is cellular, and that each celluar set is pointlike~\cite[Chapter 17]{CV1977}.

Pick sufficiently small good ball $Q\sse U$ 
such that the connected $n$--manifold \[M':=U\setminus \inte Q\] has measure larger than $r$
and has nonempty boundary. Applying Lemma~\ref{lem:brown-measure2} to $M'$, we have an open embedding
\[
f\co \bH^n_+\longrightarrow M'
\]
such that \[f(\partial \bH^n_+)\sse \partial M'=(\partial M\cap U)\cup\partial Q.\]
Since $\inte \bH^n_+$ is a countable increasing union of collared balls,
we can find a collared ball $\hat B$ in $M'$ having measure larger than $r$;
moreover, we can further require that $\hat B$ is good by countable additivity of $\mu$.
Applying Theorem~\ref{thm:ou} to $\hat B$, we see that the restriction of $\hat B$ is conjugate to a Lebesgue measure on a cube.
It is then trivial to find a good ball $B\sse \hat B$ with measure precisely $r$.
\end{proof}


\subsection{Regular open sets and homeomorphism groups}
Let $X$ be a topological space. 
If $A\subseteq X$ is a subset then we write $\cl A$ and $\inte A$ for its closure and interior, respectively, and
    \[
    \fr(A):=\cl A\setminus\inte A\] for the \emph{frontier} of $A$.
    
A set $U\sse X$ is \emph{regular open} if $U=\inte\cl U$.
For instance, a good ball is always regular open.
The set of regular open subsets of $X$ forms a Boolean algebra, denoted as $\ro(M)$.
In this Boolean structure, 
the minimal and maximal elements are the empty set and $X$ respectively.
The meet is the intersection,
and the join of two regular open sets $U$ and $V$ is given by \[U\oplus V:=\inte\cl(U\cup V).\]
We write \[U_1\sqcup U_2=V\] when $V$ is the disjoint union of two sets $U_1$ and $U_2$.

The complement coincides with the exterior: \[U^{\perp}:=X\setminus \cl U.\] 
Consequently, the Boolean partial order $U\le V$ coincides with the inclusion $U\sse V$ for $U,V\in \ro(X)$. For each subcollection $\FF\sse \ro(X)$ of regular open sets we can define its supremum as
\[ \sup \FF :=\inte\cl\left( \bigcup \FF\right)\in\ro(X).\]
In particular, $\ro(X)$ is a complete Boolean algebra.
We remark that the collection of open sets of a manifold (or indeed of an arbitrary topological space) is not a Boolean algebra in a natural way, but rather a Heyting algebra, since it is possible that $U\subsetneq U^{\perp\perp}$.

By a \emph{regular open cover} of a space, we mean a cover consisting of regular open sets. We will repeatedly use the following straightforward fact, which implies that every finite open cover of a normal space can be refined by an open cover which consists of regular open sets.

    
\begin{lem}\label{lem:refine}
If $\UU=\{U_1,\ldots,U_m\}$ is an open cover of a normal space, then there exists a regular open cover $\VV=\{V_1,\ldots,V_m\}$ such that $\cl V_i\sse U_i$ for each $i$.
\end{lem}
\begin{proof}
Under the given hypothesis, one can find an open cover $\{W_i\}$ satisfying $\cl W_i\sse U_i$ for each $i$; see~\cite[Corollary 1.6.4]{Coornaert2005}.
It then suffices for us to take $V_i:=\inte\cl W_i$, which is clearly a regular open set.
\end{proof}

Let $g\in\Homeo(X)$.
We denote its fixed point set by $\fix g$, and define its \emph{(open) support} as $\supp g:=X\setminus \fix g$.
We then define its \emph{extended support} as
\[
\suppe g:= \inte\cl \supp g = \inte \cl (X\setminus\fix g).
\]
Let $G\le\Homeo(X)$.
We define
the \emph{rigid stabilizer (group)} of $A\sse X$ as
\[
G[A]:=\{g\in G\mid  \supp g\sse A\}.\]
If $U$ is regular open in $X$, we note that
\[G[U]=\{g\in G\mid \suppe g\sse U\}.
\]

Recall from the introduction that the group $G\le\Homeo(X)$ is \emph{locally dense} 
if for each nonempty open set $U$ and for each $p\in U$ we have hat
\[
\inte\cl \left(G[U].p\right)\ne\varnothing.
\]
More weakly, we say $G$ is \emph{locally moving} if the rigid stabilizer of each nonempty open set is nontrivial. 

If $G$ is a locally moving group of homeomorphisms of $X$ then $\ro(X)$ has no atoms, and the set of extended supports \[ \{\suppe g\mid g\in G\}\] is \emph{dense} in the complete Boolean algebra $\ro(X)$, i.e.~ for all $U\in\ro(X)$ there exists $g\in G$ such that $\suppe g\subseteq U$; see \cite{Rubin1996} and \cite[Theorem 3.6.11]{KK2021book}. 
When the ambient space is a manifold,
the fundamental observation is that every regular open set can actually be represented as the extended support of some homeomorphism.

\begin{prop}\label{prop:full-support}
Suppose that $(M,G)\in\cM$  with $\dim M\ge1$,
or that $(M,G)\in\cM_{\vol}$ with $\dim M>1$.
Then each regular open set of $M$ is the extended support of some element of $G$.
\end{prop}
\begin{proof}
Pick a countable dense subset $\{x_i\}_{i\in\omega}$ of $U\cap\inte M$. 
Set $j_1:=1$, and pick a good ball $B_1$ containing $x_{j_1}=x_1$ such that $\diam B_1<1$ and such that $B_1\sse U$.
Suppose we have constructed a sequence \[j_1<j_2<\cdots j_k,\] and a disjoint collection of good balls $B_1,\ldots,B_k$ 
such that  $x_{j_i}\in B_i$ and such that \[\diam B_i<1/i\] for each $i$;
furthermore, we require that \[\{x_1,x_2,\ldots,x_{j_k}\}\sse B_1\cup \cdots \cup B_k\sse U.\]
If \[U=\bigcup_{i=1}^k B_i,\] then we terminate the procedure; 
otherwise,  we let $j_{k+1}$ be the minimal index $j$ such that \[x_j\in W:=U\setminus\bigcup_{i=1}^k B_i.\]
Pick a good ball $B_{k+1}\sse W$ containing $x_{j_{k+1}}$ such that \[\diam B_{k+1}<1/(k+1).\]
Thus, we build an infinite disjoint collection of good balls  $\{B_i\}_{i\in\omega}$ in $U$ such that \[\{x_i\}_{i\in\omega}\sse \bigcup_i B_i.\]

We claim that there exists $h_i\in G$ for each $i$ such that $\suppe h_i=\inte B_i$.
In the case where there is no measure under consideration, this is clear from the definition of a good ball. 
In the case when a measure $\mu$ is part of the data,
we first pick a homeomorphism $h$ in $\Homeo_{0,\on{Leb}}(B^n(1), \partial B^n(1))$ whose fixed point set has empty interior; here, the condition that $\dim M>1$ is used.
Let us also pick a homeomorphism \[u_i\co B^n(1)\longrightarrow B_i.\] 

We see from Theorem~\ref{thm:ou} that the pullback measure of $\mu$ on $B^n(1)$ under the map $u_i$ is conjugate to (a rescaling of) the Lebesgue measure by a homeomorphism. Hence, by conjugation and extension by the identity, we obtain a homeomorphism $h_i\in\Homeo_{0,\mu}(M)$ satisfying 
\[\fix h_i=(M\setminus\inte B_i)\sqcup Q_i\]
for some closed set $Q_i\sse B_i$ with empty interior. 
This proves the claim.

Since we have \[\sup_{x}d(x,h_i(x))\le \diam B_i<1/i\] for all $i$, we see from the uniform convergence theorem that the infinite product $g:=\prod_i h_i$ converges in $\Homeo(M)$, and is isotopic to the identity. 
By definition, \[\suppe g = \inte\cl \left(\bigcup_i\inte B_i\right) = U.\]
Hence, this map $g$ satisfies the conclusion.
    \end{proof}

Note that measure-preserving homeomorphism groups of compact one--manifolds are highly restricted.
\begin{prop}\label{prop:dim1}
For each compact connected one--manifold $M$, 
there exist a group theoretic formula $\phi_M^{\vol}$ such that 
when $(N,H)\in\cM_{\vol}$,
we have that \[ H\models \phi_M^{\vol} \] if and only if $N$ and $M$ are homeomorphic.
\end{prop}
\begin{proof}
Since $\Homeo_{\mu}(I)\cong\bZ/2\bZ$, the group theoretic sentence $\phi_I^{\vol}$ stating that there are at most two elements in the group is satisfied by a pair $(N,H)\in\cM_{\vol}$ if and only if $N\cong I$.
Since $\Homeo_{\mu}(S^1)$ contains the abelian group $\Homeo_{0,\mu}(S^1) \cong \on{SO}(2,\bR)$ as the index--two subgroup, a pair of the form $(S^1,G)\in\cM_{\vol}$ satisfies the sentence \[\phi_{S^1}^{\vol}:=(\forall\gamma_1,\gamma_2)
[\gamma_1^2\gamma_2^2\gamma_1^{-2}\gamma_2^{-2}=1]\wedge\neg\phi_I^{\vol}.\] Finally, if $(N,H)\in\cM_{\vol}$ with $\dim N>1$, then $H$ is not virtually abelian and hence $H$ does not satisfy the above formulae.
\end{proof}

\subsection{First order logic}\label{ss:model theory}
Proposition~\ref{prop:dim1} establishes the measure-preserving case of the main theorem with $\dim M=1$. 
Our strategy for all the other cases is to build a new language, one which is powerful enough that it can distinguish a given manifold from the other ones,
but which can still be ``interpreted'' to the language of groups. In order to do this, let us begin with  a brief review of the basic terminology from multi-sorted first order logic. Details can be found in~\cite{marker-book,tent-ziegler}
and also succinctly in~\cite{BH2017}.

On the syntactic side, a (multi-sorted, first order) \emph{language} $\LL$ is specified by \emph{logical symbols} and a \emph{signature}.
Logical symbols include quantifiers ($\forall$, $\exists$), logical connectives ($\wedge$, $\vee$, $\neg$, $\to$), the equality ($=$) and a  countable set of variables.
We often write auxiliary symbols such as parentheses or commas for the convenience of the reader.

A \emph{signature} consists of  \emph{sort symbols},
\emph{relation symbols} (also called as \emph{predicate symbols}), 
\emph{function symbols} and
 \emph{constant symbols}.
For the brevity of exposition we often regard a function or constant symbol as a special case of a relation symbol.
An \emph{arity function} is also in the signature,
which assigns a finite tuple of sort symbols to each relation symbol. The arity function for each constant symbol is further required to assign only a single (i.e. 1--tuple of) sort symbol.


A \emph{(well-formed) $\LL$--formula} is a juxtaposition of the above symbols which is ``valid''; the precise meaning of this validity requires a recursive definition~\cite{marker-book}, although it is intuitively clear.
For instance, if $P$ is a relation symbol with the arity value $(s,t)$ for some sort symbols $s$ and $t$,
and if $x$ and $y$ are variables with sort values $s$ and $t$, respectively,
then $Pxy$ is a formula. We write $P(x,y)$ instead of $Pxy$ for the ease of reading.
The \emph{language} $\LL$ specified by the above information is the collection of all formulae.
Unquantified variables in a formula are called $\emph{free}$, and a \emph{sentence} is a formula with no free variables.

On the semantic side, we have an $\LL$--\emph{structure} (or, an $\LL$--\emph{model}) $\mathcal{X}$, which is specified by a set $|\XX|$ called the \emph{universe}, a sort function $\sigma$ from $|\XX|$ to the set of sort symbols,
and an \emph{assignment} that is a correspondence from each relation symbol $P$ to an actual relation $P^\XX$ among tuples of the elements in the universe.
For each sort symbol $s$, we call $s^\XX:=\sigma^{-1}(s)$  the \emph{domain} of $s$ in $\XX$.
It is required that the relation $P^\XX$ respects the arity value of $P$.
For instance, if $P$ is as in the previous paragraph, then $P^\XX$ will be a subset of $s^{\XX}\times t^{\XX}$.
A function symbol is assigned the graph of some function, and often written as a function notation such as $f(x)=y$. 
A constant symbol is fixed as an element in the universe by an assignment.
An assignment (for relations) naturally extends to an assignment $\phi^\XX$ for each formula $\phi$.
We sometimes omit $\XX$ from $\phi^\XX$ when the meaning is clear.

For an $\LL$--formula $\phi$ with a tuple of free variables $\underbar x$,
and 
for a tuple $\underbar a$ of elements in $|\XX|$,
we write $\XX\models \phi(\underbar a)$
if  $\phi^\XX$ holds after $\underbar a$ 
has been substituted for
$\underbar x$.
We define $\Th(\XX)$ as the set of all $\LL$--sentences $\phi$ such that $\XX\models \phi$.

Let $p,q\ge0$, and let $\underbar b$ be a $q$--tuple of elements of $|\XX|$.
A subset $A$ of $X^p$ is \emph{definable (by $\phi$) with parameters $\underbar b=(b_1,\ldots,b_q)$}
if for some formula $\phi$ with $p+q$ free variables, the set $A$ coincides with the set
\[
\XX(\phi;\underbar b)
:=\{\underbar a\in |\XX|^p\co \phi(\underbar a,\underbar b)\}.
\]
If $q=0$ we simply say $A$ is \emph{definable}, in which case we denote the above set as $\XX(\phi)$.
We now formalize the concept of ``interpreting'' 
a new language.

\bd\label{d:interpret}
Let $\LL_1$ and $\LL_2$ be languages.
Suppose we have a class $\XX$ 
of ordered pairs in the form  $(X_1,X_2)$ with $X_i$ 
 being an  $\LL_i$--structure.
We say $X_2$  is \emph{interpretable in $X_1$ uniformly for $(X_1,X_2)$ in $\XX$} if 
there exist some $\LL_1$--formulae $\phi_{\mathrm{dom}}$ and $\phi_{\mathrm{eq}}$,
and there also exists a map $\alpha$ from the set of $\LL_2$--formulae to the set of $\LL_1$--formulae such that the following hold.
\begin{quote}
for each $(X_1,X_2)\in\XX$,
we have a surjection
\[
\rho\co X_1(\phi_{\on{dom}}) \longrightarrow |X_2|\]
with its fiber uniformly defined by $\phi_{\on{eq}}$ in the sense that
\[X_1(\phi_{\on{eq}})=\{(x,y)\in X_1(\phi_{\on{dom}})\times X_1(\phi_{\on{dom}})\mid \rho(x)=\rho(y)\}.\]
Furthermore, it is required  
for each $\LL_2$--formula $\psi$
that
\[
\rho^{-1}(X_2(\psi))=X_1(\alpha(\psi)).
\]    
\end{quote}
The bijection 
\[
\rho^{-1}\co |X_2|\longrightarrow X_1(\phi_{\on{dom}})/X_1(\phi_{\on{eq}})\]
along with the map $\alpha$ is called a \emph{uniform interpretation} of $X_2$ in $X_1$.
\ed

\begin{rem}\label{rem:interpret}
\be
\item\label{p:interpret-rem}
In the above, if $\psi$ is $m$--ary (as a relation) and $\phi_{\on{dom}}$ is $n$--ary, then $\alpha(\psi)$ is $mn$--ary.
In practice, we only need to consider relation symbols 
(in a broad sense, including function and constant symbols)
$\psi$ rather than all possible $\LL_2$--formulae.
\item\label{p:interpret-rho} In various instances of this paper, it will be the case that $\LL_1\sse \LL_2$
and that the interpretation restricts to the identity on $\LL_1$. 
As a consequence of such interpretability, we will have that $\Th(X_2)$ is a conservative extension of $\Th(X_1)$ for each $(X_1,X_2)\in\XX$.
Also, we will often add a function symbol in $\LL_2$ corresponding to the surjection $\rho$, which is clearly justified.
\ee
\end{rem}

The following lemma explains how the combination of Theorems~\ref{thm:interpret}
and~\ref{thm:iamM} implies Theorem~\ref{thm:main}.
\begin{lem}\label{lem:inter-models}
Suppose $\LL_1,\LL_2$ and $\XX$ are as in Definition~\ref{d:interpret} so that $X_2$ is  interpretable in $X_1$ uniformly for $(X_1,X_2)\in\XX$.
Let $(X_1,X_2)$ and $(Y_1,Y_2)$ be in $\XX$.
Then for each sentence $\psi$ belonging to $\Th(X_2)\setminus \Th(Y_2)$, the interpretation $\alpha(\psi)$ belongs to $\Th(X_1)\setminus \Th(Y_1)$.
In particular,  if $X_1\equiv Y_1$, then $X_2\equiv Y_2$.
\end{lem}

\section{The AGAPE structure and basic observations}\label{sec:sorts}  
The fundamental universe that we work in will be the group of homeomorphisms of a manifold. 
Objects such as regular open sets, real numbers, points in the manifold, continuous functions, etc.~will all be constructed as definable equivalence classes of definable subsets of finite tuples of homeomorphisms.

\subsection{The langauge \texorpdfstring{$L_{\agape}$}{L-AGAPE} and the structure \texorpdfstring{$\agape(M,G)$}{AGAPE(M,G)}}
The ultimate language we will work in will be called $\agape$, which stands for ``Action of a Group, Analysis, Points and Exponentiation''.
This language is denoted as $L_{\agape}$ and contains the following different sort symbols for $k,\ell\in\omega$:
\[\bG,\ro,\bN,\PP(\bN),\bR,\bM,\bM^{\on{disc-int}},\cont_{k,\ell}.\]
The above sorts come with some symbols that are intrinsic to the sort (such as a group operation),
and others which relate the sorts to each other,
 as we spell out below. 
There will be a countable set of variables for each sort, as is typically required.
We also describe an $\agape$ structure assigned to each pair $(M,G)$ in the class $\cM$ or $\cM_{\vol}$.
In this structure,
we give the ``intended'' choice of the {domain of each sort symbol.}

{\bf The group sort.} 
The domain of the sort symbol $\bG$ will be the group $G$, under our standing assumption that $(M,G)\in\cM_{(\vol)}$. 
The signatures only relevant for this sort are 
\[1,\circ,{}^{-1},\]
which are respectively assigned with the natural meanings in the group theory.
These symbols, along with variables, form the \emph{language of groups} $L^0_{\act}=L_\bG$. The group $G$ is regarded an $L_{\bG}$-structure $\act^0(M,G)=\act_{\bG}(M,G)$. We will usually not write
the $\circ$ symbol.

{\bf The sort of regular open sets.} 
The domain of the sort symbol $\ro$ is the set $\ro(M)$ of the regular open sets in $M$.
The newly introduced signatures for this sort are
\[
\sse,\cap,\,^{\perp},\oplus,\varnothing,\bM,\suppe,\appl.
\]
The symbol $\bM$ means the manifold $M$ in the structure.
By the natural assignment as before, we have Boolean symbols \[\sse,\cap,\oplus,\,^\perp,\bM, \varnothing\]
for the Boolean algebra $\ro(M)$.
We let the function symbol $\suppe$ mean the map $G\longrightarrow \ro(M)$ defined as
\[g\mapsto \suppe g.\]
We have an assignment for $\appl$ so that \[\appl(g,U)=g(U)\] with $g\in G$ and $U\in\ro(M)$.
The symbols introduce so far (along with countably many variables for each sort) form the
\emph{language of a group action on a Boolean algebra}
$L^1_{\act}=L_{\mathbf{G},\ro}$.
The $L^1_{\act}$--structure described above on the universe $G\sqcup\ro(M)$ is denoted as $\act^1(M,G)=\act_{\bf G,\ro}(M,G)$.

{\bf The sorts from the analysis}
We then introduce new sort symbols, which are $\bN, \PP(\bN),\bR$ and $\cong_{k,\ell}$ for $k,\ell\in\omega$.
The signatures introduced here are
\[
0,1,<,+,\times,\in,\subseteq,\#\pi_0,\on{norm}.\]
Standard second order arithmetic
\[\on{Arith}_2=(\bN,\PP (\bN),0,1,<,+,\times,\in,\subseteq)\] 
is given the sort symbols $\bN$ and $\PP(\bN)$, as well as with relevant non-logical symbols.
We note the ambiguity of our notation that the sort symbols $\bN$ and $\PP(\bN)$ will be assigned with the set of the natural numbers $\bN$ and its power set $\PP(\bN)$, respectively.
The symbol $\#\pi_0$ is interpreted so that
\[\#\pi_0(U)=k\]
means $U\in\ro(M)$ has $k$ connected components.
See Section~\ref{sec:arith} for details. 
The ordered ring of the real numbers
\[\{0,1,+,\times,<,=\}\]
is assigned with the sort symbol $\bR$ and the signatures above. 
Note that, 
as is usual, $\bN$ is considered as a subsort of $\bR$, by identifying each integer as a real number.

The domain of the sort symbol $\cont_{k,\ell}$ will be the set $C(\bR^k,\bR^\ell)$ of continuous functions.
We also have a formula
$\appl(f,x)=y$ when the sort value of $f$ is $\cont_{k,\ell}$,
and when $x$ and $y$ are tuples of variables assigned with the sort symbol $\bR$.
We have the $C^0$--norm \[f\mapsto \on{norm}(f):=\|f\|,\] which will be also a part of the language.
Combining these symbols with $L^1_{\act}$, we obtain the language 
$L^2_{\act}=L_{\bG,\ro,\bR}$.
An $L^2_{\act}$--structure
$\act^2(M,G)=\act_{\bG,\ro,\bR}$ is assigned to  each $(M,G)\in\cM_{(\vol)}$ having the universe 
\[G \sqcup\ro(M)\sqcup\PP(\bN)\sqcup\bR\sqcup \bigsqcup_{k,\ell}C(\bR^k,\bR^\ell).\]

{\bf The point and the discrete subset sorts $\bM$ and $\bM^{\on{disc-int}}$.}
The domain of the sort symbol $\bM$ will be the set of the points in a manifold.
We also introduce the sort symbol $\bM^{\on{disc-int}}$ to mean a subset $A$ of $\inte M$ every point of which is isolated in $A$.
By abuse of notation, the symbols $\in$ and $\sse$ introduced above will have multiple meanings (depending on the context), so that they have the arity values $(\bM,\ro)$, $(\bM,\bM^{\on{disc-int}})$ and $(\bM^{\on{disc-int}},\ro)$.

We also have a cardinality function
\[
\# A=m\]
meaning that the cardinality of $A\sse \inte M$ is $m$, assuming that every point in $A$ is isolated.

The interpretation of points of the manifold will allow us to include symbols such as $\cl$ and $\fr$, the closure and frontier of a regular open set, together with membership relations into these sets. These symbols will simply be abbreviations for formulae which impose the intended meaning.
We will be able to separate out boundary points of $M$ from the interior ones, and hence justified to use the notations
\[
\pi\in\partial \bM, \pi'\in\inte \bM\]
for point sort variables $\pi$ and $\pi'$.
The function symbol $\appl$ has a natural additional meaning as below:
\[\appl\colon G\times M\longrightarrow M.\]
In all contexts, we abbreviate $\appl(\gamma,x)$ by $\gamma(x)$
when the sort of $\gamma$ is either $\bG$ or $\cont_{k,\ell}$
and when the sort of $x$ is (tuples of) $\bR$, $\bM,\bM^{\on{disc-int}}$ or $\ro$.

The omnibus language, combining all of the previous sorts and relevant symbols, is denoted by \[L^3_{\act}=L_{\bG,\ro,\bR,\bM}=L_{\agape},\] or simply as $\agape$.
We have so far described the $L_{\agape}$--structure $\act^3(M,G)=\agape(M,G)$ corresponding to $(M,G)\in\cM_{(\vol)}$.

Dealing with these structures, we  often make use of functions or relations defined by fixed formulae that are not explicitly specified. The following terminology will be handy when we need to avoid  ambiguity in such situations: 

\bd\label{defn:unif-def}
Let $i=0,1,2,3$,
and let  $\phi_{(\vol)}$ be a formula in $L^i_{\act}$.
Suppose for each  $(M,G)\in\cM_{(\vol)}$ that a function or relation $f_{M,G}$  is defined by $\phi$ in $\act^i(M,G)$.
Then the collection 
\[
\{f_{M,G}\mid (M,G)\in\cM_{(\vol)}\}\]
is said to be \emph{uniformly defined over $\cM_{(\vol)}$}.
\ed

\begin{rem}\label{rem:notation}
In dealing with the sorts in Subsection~\ref{ss:model theory}, we will distinguish notationally between variables referring to a particular sort and elements of that sort. For the convenience of the reader, we will record a table summarizing the notation. In general, we will write an underline to denote an arbitrary (or simply
unspecified) finite tuple of variables or objects.
    \begin{center}
    	\begin{tabular}{ | l | l | l | l | p{3cm} |}
    		\hline
    		Sort & variable & object  \\ \hline
    		Group elements & $\gamma$, $\delta$, $\underline\gamma$, $\underline\delta$ & $g$, $h$  \\ \hline
    		Regular open sets & $u$, $v$, $w$, $\underline u$, $\underline v$, $\underline w$ & $U$, $V$, $W$ \\ \hline
    		Natural numbers & $\alpha$, $\beta$, $\underline\alpha$, $\underline\beta$ & $k,m,n$\\ \hline
    		Sets of natural numbers & $\Lambda$, $\underline\Lambda$ & $A$  \\ \hline
    		Real numbers & $\rho$, $\sigma$, $\underline\rho$, $\underline\sigma$ & $r$, $s$ \\
    		\hline
    		Sets of points & $\pi$, $\underline\pi$, $\tau$, $\underline\tau$ & $p$, $q$, $T$ \\
    		\hline
    		Functions & $\chi$, $\theta$, $\underline\chi$, $\underline\theta$ & $f$ \\
    		\hline
    	\end{tabular}
    \end{center}
From now on, we will reserve the letters in this table for exclusive use as variables or objects of a particular sort, unless specified otherwise. 
In the ambient metalanguage,
we will use $i,j\in\omega$ to denote indices. The symbols $M$ and $N$ will be reserved for manifolds.
\end{rem}

\subsection{Interpreting action structures in homeomorphism groups}
Since the uniform interpretability (Definition~\ref{d:interpret}) is transitive, 
the following proposition would trivially imply Theorem~\ref{thm:interpret}.

\begin{prop}\label{prop:induct}
For each $i=0,1,2$, 
and uniformly for $(M,G)\in\cM_{(\vol)}$,
the $L^{i+1}_{\act}$--structure $\act^{i+1}(M,G)$ is interpretable in
the $L^{i}_{\act}$--structure $\act^{i}(M,G)$.
\end{prop}

The proof of this proposition will require the construction of 
 $L_{\act}^i$--formulae $\phi_{\on{dom}}^i$ and $\phi_{\on{eq}}^i$,
and a surjection
\[
\rho_i\co \act^{i}(M,G)(\phi_{\on{dom}}^i)\longrightarrow |\act^{i+1}(M,G)|\]
for all $(M,G)\in\cM_{(\vol)}$ satisfying the conditions of Definition~\ref{d:interpret}.
Our construction will occupy Sections~\ref{sec:arith} and~\ref{sec:points}, as well as most of this section.

    
 Rubin's Theorem~\cite{Rubin1989,Rubin1996} stated in the introduction
 can be used to prove various \emph{reconstruction theorems}, by which we mean that
    group isomorphism types greatly restrict the homeomorphism types of spaces on which groups
    can act nicely. See~\cite{KK2021book} for comprehensive references on this, especially regarding diffeomorphism groups.
    
    A key step in the proof of Rubin's theorem can be rephrased as follows. We emphasize that the formulae below are independent of the choice of the group $G$ or the space $X$. 

\begin{thm}[Rubin's Expressibility Theorem, {cf.~\cite{Rubin1996}}]\label{t:express}
There exist first order formulae \[ \subseteq(\gamma_1,\gamma_2),\quad \appl(\gamma_1,\gamma_2,\gamma_3),\quad\cap(\gamma_1,\gamma_2,\gamma_3), \quad\oplus(\gamma_1,\gamma_2,\gamma_3),\quad\perp(\gamma_1,\gamma_2)\] in the language of groups such that if $G$ be a locally moving group of homeomorphisms of a Hausdorff topological space $X$, then the following hold for all $g_1,g_2,g_3\in G$.
\be
\item
  	$\displaystyle
	G\models\sse(g_1,g_2)\Longleftrightarrow \suppe g_1 \sse \suppe g_2
    	\Longleftrightarrow g_1\in G[\suppe g_2]$.
    	\item
    	$\displaystyle
    	G\models\mathrm{appl} (g_1,g_2,g_3)\Longleftrightarrow \appl(g_1, \suppe g_2) =\suppe g_3$.
    	\item
    	$\displaystyle
    	G\models \cap(g_1,g_2,g_3)\Longleftrightarrow \suppe g_1 \cap \suppe g_2 =\suppe g_3$.
    	\item
    	$\displaystyle
    	G\models \oplus (g_1,g_2,g_3)\Longleftrightarrow \suppe g_1 \oplus \suppe g_2 =\suppe g_3$.
    	\item
    	$\displaystyle
    	G\models \ext(g_1,g_2)\Longleftrightarrow \suppe g_1 =(\suppe g_2)^{\perp}$.
    	\ee
    \end{thm}
\begin{proof}
Parts (1) and (2) are given as Theorem 2.5 in~\cite{Rubin1989}; see also~\cite[Corollary 3.6.9]{KK2021book} for a concrete formula. The remaining items are clear from the fact that the supremum in $\ro(M)$ is first order expressible in terms of the inclusion relation.
\end{proof}
Let $(M,G)\in\cM_{\vol}$.
By Proposition~\ref{prop:full-support}, we have a surjection 
\[\rho_0\co G\longrightarrow\ro(M)\] defined as $g\mapsto \suppe g$. 
Since $G$ is locally moving on $M$, 
Rubin's expressibility theorem implies that
the fiber \[\{(g,h)\mid \suppe g = \suppe h\}\]
of $\rho_0$ is definable,
and that the Boolean symbols and the function symbols $\appl$ and $\suppe$ have group theoretic interpretations; see also parts (\ref{p:interpret-rem}) and (\ref{p:interpret-rho}) of Remark~\ref{rem:interpret}.
We conclude the following, which shows that Proposition~\ref{prop:induct} holds for the case $i=0$.
\begin{cor}\label{cor:interp-action}
Uniformly for $(M,G)\in\cM_{(\vol)}$,
the $L_{\bf G,\ro}$--structure $\act_{\bf G,\ro}(M,G)$ is interpretable in 
the group structure $G$.
\end{cor}

Corollary~\ref{cor:interp-action} can be summarized as saying that $ G$ interprets the group action structure of $ G$ on the algebra of regular open sets, in a way that preserves the meaning of $ G$.    This interpretation is uniform in the underlying pair $(M,G)$, and any formula in the language of $G$ and $\ro$ can be expressed entirely in $G$, since the formulae in Theorem~\ref{t:express} are independent of $M$. Henceforth, we will assume that we work in the expanded language $L_{\act}^1=L_{\bG,\ro}$.

\subsection{First order descriptions of basic topological properties}\label{ss:top}
Recall that whenever the expression $U\sqcup V$ is used it is assumed that $U$ and $V$ are disjoint.

We now produce first order expressions for some standard point-set--topological properties. 

\begin{lem}\label{lem:top-prop}
The following hold for $(M,G)\in\cM_{(\vol)}$.
\be
\item\label{p:full1}
For $U,V\in\ro(M)$, we have that $ G[U]= G[V]$ if and only if $U=V$.
\item\label{p:rigid-stab}
For each $U\in\ro(M)$, we have that
\[
G[U]=\{g\in G\mid g(V)=V\text{ for all regular open set }V\sse U^\perp\}.
\]
\item\label{p:pcnt}
An open subset is path-connected if and only if it is connected. 
\item\label{p:cnt1}
An arbitrary union of connected components of a regular open set is necessarily regular open.
More specifically, if a regular open set $W$ can be written as $W=U\sqcup V$ for some disjoint pair of open sets $U$ and $V$, then $U$ and $V$ are regular open and $W=U\oplus V$.
Moreover, we have $V=W\cap U^{\perp}$.
\item\label{p:group}
For disjoint pair $U,V$ of regular open sets, we have (i)$\Rightarrow$(ii)$\Rightarrow$(iii).
\be[(i)]
\item $U$ is connected, and $U\oplus V=U\sqcup V$;
\item Every $g\in  G[U\oplus V]$ satisfies either $g(U)=U$ or $g(U)\cap U=\varnothing$;
\item $U\oplus V=U\sqcup V$
\ee
\item\label{p:cnt2}
    	Let $W$ and $U$ are regular open sets such that $U$ is connected and such that $U\sse W$.
    	Then $U$ is a connected component of $W$
    	if and only if $W=U\oplus V$ for some regular open $V$ that is disjoint from $U$,
    	and every $g\in  G[W]$ satisfies either $g(U)=U$ or $g(U)\cap U=\varnothing$.
    	\item\label{p:cnt1-2} The following are all equivalent for a regular open set $W$.
    	\be[(i)]
    	\item  $W$ is disconnected;
    	\item $W=U\sqcup V$ for some disjoint pair of nonempty regular open sets $U$ and $V$ such that $U$ is connected;
    	\item $W=U\oplus V$ for some disjoint pair of nonempty regular open sets $U$ and $V$, and every $g\in  G[W]$ satisfies either $g(U)=U$ or $g(U)\cap U=\varnothing$;
    	\item $W=U\oplus V=U\sqcup V$ for some disjoint pair of nonempty regular open sets $U$ and $V$.\ee
    \item\label{p:full2}
    For two regular open subsets $U$ and $V$
    satisfying $U\cap V=\varnothing$,
    we have that
    $U\sqcup V= U\oplus V$
    if and only if
    each connected component of $U\oplus V$ is contained either in $U$ or in $V$.
    \ee
    \end{lem}
    \begin{proof}
    (\ref{p:full1})
    If
    $x\in U\setminus V$, then 
    there exists some $h\in  G[U]$
    satisfying $h(x)\ne x$; see \cite[Lemma 3.2.3]{KK2021book} for instance.
    In particular, we have \[h\in  G[U]\setminus  G[V]\ne\varnothing.\] This proves the nontrivial part of the given implication.
We remark that the same statement holds without the assumption that $U$ and $V$ are regular open, under the extra hypothesis that $M\not\cong I$.
Part (\ref{p:rigid-stab}) is similar.
    
    (\ref{p:pcnt}) This part is clear from the fact that every manifold is locally path-connected.

    (\ref{p:cnt1})
    Whenever two open sets $U$ and $V$ are disjoint
    we have that $U^{\perp\perp}$ and $V^{\perp\perp}$ are also disjoint; see~\cite[Lemma 3.6.4 (4)]{KK2021book}, for instance.
    From
    \[
    W=U\sqcup V\sse U^{\perp\perp}\sqcup V^{\perp\perp}\sse U^{\perp\perp}\oplus  V^{\perp\perp} \sse W^{\perp\perp}=W,\]
    we see that $U$ and $V$ are actually regular open and $W=U\oplus V=U\sqcup V$. It is clear that  $V\cap \fr U=\varnothing$, 
    which implies $V= W\cap U^{\perp}$.

    (\ref{p:group}) The implication (i)$\Rightarrow$(ii) is clear from that every setwise stabilizer of $g\in \cH[U\oplus V]$ permutes connected components of $U\oplus V$.
    
    For the implication (ii)$\Rightarrow$(iii), assume we have a point
    \[
    p_1\in
    (U\oplus  V)\setminus (U\cup V).\]
    Take a sufficiently small open ball $B$ around $p_1$ so that
    \[
    B\sse U\oplus V
    =
    \inte\cl (U\cup V)\sse\cl U\cup \cl V.\]
    Note also that because \[p_1\not\in U=\inte \cl U,\] it follows that $B\not\sse\cl U$. Similarly, $B\not\sse\cl V$.
    This implies that we can choose distinct points
    \[p_2,p_3\in B\cap U\cap\inte M\]
    and $p_4\in B\cap V\cap\inte M$.
Since $G$ is $k$--transitive on $B\cap \inte M$ for all $k$,
we can find a $g\in G$ supported in $B$ satisfying
    $g(p_2)=p_3$ and $g(p_3)=p_4$;
    see also  Lemma~\ref{lem:transitive}.
    Then $g(U)$ is neither $U$ nor disjoint from $U$.
    
    Parts (\ref{p:cnt2}) and (\ref{p:cnt1-2}) are clear from the previous parts.

    (\ref{p:full2}) The forward direction comes from the observation that $(U,V)$ is a disconnection of $U\oplus V$. The  backward direction is trivial since the hypothesis implies that $U\oplus V\sse U\cup V$.
    \end{proof}
    
Let us note the following consequences of Lemma~\ref{lem:top-prop}.

\begin{cor}\label{cor:basic-form}
There exist first order formulae in the language $L_{\bf{G},\ro}$ as follows:
\begin{enumerate}
\item\label{p:contained} A formula $\on{contained}(\gamma,u)$, also abbreviated as $\gamma\in\bG[u]$
such that 
\[
\models\on{contained}(g,U)
\quad\textrm{if and only if }\suppe g\sse U.\]
\item\label{p:conn}
A formula $\on{conn}(u)$ such that \[\models \on{conn}(U)\quad\textrm{if and only if $U$ is connected.}\]
\item\label{p:cc}
    	A formula $\on{cc}(u,v)$, also abbreviated as $u\in\pi_0(v)$
     such that \[\models \on{cc}(U,V)\quad\textrm{if and only if $U$ is a connected component of $V$.}\]
\item\label{p:ucc}
    	A formula $\on{ucc}(u,v)$ such that \[\models \on{ucc}(U,V)\quad\textrm{if and only if }U\textrm{ is a union of connected component of }V.\]
\item\label{p:cck}
For all $k\in\omega$, a formula $\#\on{cc}_k(u)$ such that \[\models \#\on{cc}_k(U)\quad\textrm{if and only if $U$ has exactly $k$ connected components.}\]
\item\label{p:disj} A formula $\on{disj}(u,v)$ such that
\[
\models \on{disj}(U,V)\textrm{ if and only if }U\oplus V=U\sqcup V.\]
\item\label{p:ccpartition}
A formula $\on{ccpartition}(u,v,w)$ such that
\begin{align*}
\models \on{ccpartition}(U,V,W)&\quad\text{if and only if }
\on{ucc}(U,W)\wedge\on{ucc}(V,W)\wedge W=U\sqcup V.
\end{align*}
\item\label{p:cnteq}
A formula $\#_=(u,v)$ such that
for all regular open sets $U$ and $V$ having finitely many connected components, 
we have
\[\models \#_=(U,V)\] if and only if $U$ and $V$ have the same number of connected components.
\end{enumerate}
    \end{cor}
\begin{proof}
The existence of the formula $\on{contained}(\gamma,u)$ is trivial since $\supp^e$ and $\sse$ belong to the signature of $L_{\bG,\ro}$. The formulae $\on{conn}(u)$ and $\on{cc}(u,v)$ exist by parts (\ref{p:cnt2}) and (\ref{p:cnt1-2}) of Lemma~\ref{lem:top-prop}. We can then set \[\on{ucc}(u',u)\equiv(\forall w)[\on{cc}(w,u')\rightarrow \on{cc}(w,u)].\]
The construction of the formulae $\on{cc}_k(u)$ and $\on{disj}(u,v)$ follows from the same lemma, which also implies that the formula
\[
\on{ccpartition}(u,v,w)\equiv [\on{ucc}(u,w)\wedge \on{ucc}(v,w)\wedge w=u\oplus v\wedge u\cap v=\varnothing]\]
has the meaning required in part~(\ref{p:ccpartition}).
Finally, we set
\begin{align*}
    \#_=(u,v)\equiv &(\exists u'\sse u,\exists v'\sse v)
    [
    (\forall \hat u\in\pi_0(u))[\on{conn}(u'\cap\hat u)
    \wedge u'\cap\hat u\neq\varnothing]
    \wedge\\
    &(\forall \hat v\in\pi_0(v))[\on{conn}(v'\cap\hat v)
    \wedge v'\cap\hat v\neq\varnothing]
    \wedge
    (\exists\gamma)[\gamma(u')=v']].
    \end{align*}
From the transitivity on good balls (of equal measure, in the measure preserving case) as in Lemma~\ref{lem:transitive}, we see $\#_=(u,v)$ has the intended meaning.
\end{proof}

Using the above formula, we can distinguish the case that $\dim M=1$ among all compact connected manifolds. 
\begin{cor}\label{cor:dim1}
For each compact connected one--manifold $M$, 
there exist $L_{\bf G,\ro}$--formulae $\phi_M$ such that when $(N,H)\in\cM$,
 we have that \[ \act_{\bf G,\ro}(N,H)\models \phi_M \] if and only if $N$ and $M$ are homeomorphic.
\end{cor}
\begin{proof}
We let $\phi_I$ be the $L_{\bf G,\ro}$--formula expressing that for all pairwise disjoint, proper, nonempty regular open sets $U_1,U_2$ and $U_3$ the exterior of $U_i$ is disconnected for some $i$. This formula holds for $M\cong I$ since at least one of $\cl U_i$ does not intersect $\partial M$, and hence $U_i^\perp$ separates the two endpoints of $M$. It is clear that $\phi_I$ is never satisfied by other compact connected manifolds. 

We now suppose that $M$ is a compact connected manifold not homeomorphic to $I$. Then for all disjoint, proper, non-empty regular open sets $U$ and $V$ satisfying $U\oplus V=U\sqcup V$,
the set $M\setminus (U\oplus V)$ is disconnected. From Corollary~\ref{cor:basic-form}, we obtain the formula $\phi_{S^1}$ expressing that $M$ and $S^1$ are homeomorphic.
\end{proof}

By Proposition~\ref{prop:dim1} and Corollary~\ref{cor:dim1}, we establish Theorem~\ref{thm:main} for the case when $M$ is one--dimensional.
\emph{Henceforth, we modify the definitions of $\cM$ and $\cM_{\vol}$, replacing these classes by subclasses where all the manifolds in consideration are of dimension at least two.}

\subsection{Further topological properties}
We will need several more general first order formulae to express basic topological properties of regular open sets. One of primary importance will be a formula which implies that a particular regular open set $U$ is contained in a collared ball inside of another regular open set $V$. This is not particularly difficult to state and prove in the class $\cM$, but is substantially harder in $\cM_{\vol}$. For the rest of this section, we will make a standing assumption that $(M,G)\in\cM_{(\vol)}$, and that the underlying structure is $\act_{\bG,\ro}(M,G)$.
    
\subsubsection{Relative-compactness regarding good balls}\label{ss:more-topol}
We use the preceding results to find first order formulae that compare measures of regular open sets.
For the remainder of this subsection, we assume that $M$ is a connected, compact $n$--manifold with $n>1$, equipped with an Oxtoby--Ulam measure $\mu$.
    
\begin{lem}\label{lem:leq-fo}
There exists a formula $\vol_{\leq}(u_1,u_2,v)$ in the language $L_{\bf{G},\ro}$ such that 
for all $(M,G)\in\cM_{\vol}$ with an Oxtoby--Ulam measure $\mu$ on $M$, 
and for all any triple $(U_1,U_2,V)$ with $U_1$ and $U_2$ connected and $U_1,U_2\subseteq V$, we have the following:
    \begin{enumerate}
    	\item
    	If $\mu(U_1)\leq \mu(U_2)$ then $\vol_{\leq}(U_1,U_2,V)$ holds.
    	\item
    	If $\vol_{\leq}(U_1,U_2,V)$ holds then $\mu(U_1)\leq\mu((\cl U_2)\cap V)$.
    \end{enumerate}
    \end{lem}
    \begin{proof}     
Suppose first that $\mu(U_1)\le \mu(U_2)$, and let $\varnothing\ne W_0\in\ro(M)$ be arbitrary.
By Lemma~\ref{lem:measure-ball}, we can find a good  ball $B\sse U_1$ such that
\[
\mu(U_1)-\mu(W_0)<\mu(B)<\mu(U_1)\le\mu(U_2),
\]
and such that $U_1\setminus B$ connected. 
Lemma~\ref{lem:transitive} furnishes  $g\in\Homeo_{0,\mu}(M)[V]$ such that $g(B)\sse U_2$, but clearly there is no $\mu$-preserving $h$ such that \[h(W_0)\subseteq U_1\setminus B=U_1\cap (\inte B)^\perp.\] 
We have just established that $\vol_{\leq}(U_1,U_2,V)$ holds with $W_1:=\inte B$, where
\begin{align*}
    \vol_{\leq}(u_1,u_2,v)\equiv&
    (\forall w_0\ne\varnothing,\exists w_1)[w_1\sse u_1\wedge 
    \on{conn}(u_1\cap w_1^{\perp})\wedge \\ &
    	(\forall \gamma\in  \bG)[
    	\gamma(w_0)\not\subseteq u_1\cap w_1^{\perp}]\wedge
      (\exists \delta\in  \bG[v])[\delta(w_1)\subseteq u_2]].
    \end{align*}

Let us now suppose for a contradiction that $\vol_{\leq}(U_1,U_2,V)$ holds but that \[\mu(U_1)>\mu(V\cap \cl U_2).\] Let $W_0$ be the interior of a good ball in $M$ with measure $r_0<\mu(U_1)-\mu(V\cap \cl U_2)$.
It suffices to show that there is no witness $W_1$ as required by $\vol_{\leq}$.
    
If such a $W_1$ exists then by the condition on $\gamma$, we see again from Lemmas~\ref{lem:transitive} and~\ref{lem:measure-ball} that $\mu(U_1\cap W_1^\perp)\le r_0$.
Moreover, there is a group element $g\in \cH[V]$ such that $g(W_1)\subseteq U_2$, so that in fact \[g((\cl W_1)\cap V)\subseteq (\cl U_2)\cap V.\] 
We then obtain
\[
r_0= \mu(W_0) \ge \mu(U_1\cap W_1^\perp) \ge \mu(U_1)-\mu(V\cap \cl W_1)
\ge \mu(U_1)-\mu(V\cap \cl U_2).
\]
This violates the choice of $r_0$.
    \end{proof}

    The foregoing discussion allows us to characterize when a regular open set $U$ is contained in a collared ball $B$ inside
    a regular open set $V$. There are separate formulae which apply in the measure-preserving case, and in the general case.

\begin{lem}\label{lem:rcb}
There exists a first order formula $\on{RCB}_{(\vol)}(u,v)$ such that for each $(M,G)\in\cM_{(\vol)}$, we have that \[\models \on{RCB}_{(\vol)}(U,V)\]  if and only if $U$ is relatively compact in some good ball contained in $V$.
\end{lem}

Recall  our convention that this lemma
actually claims to produce two formulae, namely $\on{RCB}(u,v)$ and $\on{RCB}_{\vol}(u,v)$.

\begin{proof}[Proof of Lemma~\ref{lem:rcb}]
Let us consider the formula $\on{RCB}(u,v)$, which expresses that there exists some component $\hat v$ of $v$
satisfying the following two conditions:
\begin{itemize}
    \item $\hat v$ contains $u$;
    \item for each nonempty, regular open set $w$ contained in $\hat v$,
there exists some element $\gamma\in \mathbf{G}[\hat v]$
that moves $u$ into $w$.
\end{itemize}

We first claim that this formula satisfies the conclusion for $(M,G)\in\cM$.
Indeed, if $U$ is relatively compact in a collared ball $B\subseteq V$,
then there exists a unique $\hat V\in\pi_0(V)$ containing $B$, and hence $U$.
For each nonempty regular open $W\subseteq \hat V$, 
we see from Lemma~\ref{lem:transitive} that some $g\in G[\hat V]$ 
satisfies
\[
g(U)\sse g(B)\sse W,\]
as desired.
Conversely, suppose $\models\on{RCB}(U,V)$ holds
and let $\hat V$ be the connected component of $V$ containing $U$.
Let us fix a collared ball $B$ in $\hat V$
and set $W:=\inte B$. By assumption, we can find $g\in G[\hat V]$ such that $g(U)\sse W$.
Then 
$U$ is relatively compact in the collared ball $g^{-1}(B)$ in $V$.

For the case when $(M,G)\in\cM_{\vol}$, we set
\begin{align*}
&\on{RCB}_{\vol}(u,v)\equiv  (\exists u',\hat v)[(\hat v\in\pi_0(v)\wedge u\subseteq u'\subsetneq \hat v\wedge \on{conn}(u'))\wedge\\
&(\forall w)[(\on{conn}(w)\wedge w\subseteq \hat v
\wedge
\neg\vol_{\leq}(w,u',\hat v))\longrightarrow (\exists\gamma\in  \bG[\hat v])[\gamma(u')\subseteq w]]].
\end{align*}

In order to prove the forward direction, assume that $\on{RCB}_{\vol}(U,V)$ holds for some nonempty $U,V\in\ro(M)$. Let $U'$ and $\hat V$ be witnesses for the existentially quantified variables $u'$ and $\hat v$. 
Since $U'\subsetneq \hat V$, the Boolean subtraction $\hat V\cap (U')^\perp$ is nonempty.
We now see that 
\[ \mu(\hat V)=\mu(\hat V\cap \cl U')+\mu(\hat V\cap (U')^\perp)>\mu(\hat V\cap \cl U').\] 
So, Lemma~\ref{lem:measure-ball} furnishes a good ball $B\sse \hat V$ 
satisfying
\[\mu(B)>\mu(\hat V\cap \cl U').\] 
By Lemma~\ref{lem:leq-fo}, we have that $\neg\vol_\leq(\inte B,U',\hat V)$, and that some $g\in G[\hat V]$ satisfies that $g(U')\sse \inte B$. It follows that \[\cl U\sse\cl U'\sse g^{-1}(B)\sse \hat V\sse V,\] as desired.

For the backward direction, we pick a good ball $B$ satisfying $U\sse B\sse \hat V$ for a suitable $\hat V\in\pi_0(V)$ and set $U':=\inte B$.
Consider an arbitrary connected regular open set $W\sse \hat V$ satisfying $\neg\vol_\leq(W,U',\hat V)$. From Lemma~\ref{lem:leq-fo} again, we see that \[\mu(W)>\mu(U')=\mu(B).\] We may therefore find some $g\in G[\hat V]$ such that $g(B)\sse W$. This shows that $\on{RCB}_{\vol}(U,V)$ holds.
\end{proof}

When using Lemma~\ref{lem:rcb}, we will write $\on{RCB}$ both in the case of the full homeomorphism group and the measure-preserving homeomorphism group, suppressing the symbol $\vol$ from the notation. 

Many of the formulae below will actually have different meanings for $\cM$ and for $\cM_{\vol}$, though sometimes coincide in their implications;
we record the fact that $\on{RCB}_{(\vol)}(U,M)$ implies that $\cl U\sse \inte M$.

\subsubsection{Detecting finiteness of components}\label{ss:fin}
From part (\ref{p:cck}) of Corollary~\ref{cor:basic-form},
we can detect whether or not a given regular open set has exactly $k$ connected component in the theory of $G$ for each fixed $k\in\omega$. 
It is not obvious \emph{a priori} how to express the infinitude of the connected components of $U\in\ro(M)$,
as such an infinitude would be equivalent to the infinite conjunction 
\[\neg\cc_0(U)\wedge\neg\cc_1(U)\wedge\neg\cc_2(U)\wedge\cdots.\]
However, one can express such an infinitude in a single formula.

\begin{defn}\label{d:dispersed}
Let us set
\[
\on{dispersed}(u)\equiv (\forall \hat u\in\pi_{0}(u))[\on{RCB}(\hat u,u^{\perp}\oplus \hat u)].
\]
We say a regular open set $U$ is \emph{dispersed} if 
$\on{dispersed}(U)$ holds.
\end{defn}
Note that $\on{dispersed}(U)$ implies that
\[\cl \hat U\cap \cl (U\setminus \hat U)=\varnothing\]
for each connected component $\hat U$ of $U$.
Let us introduce another formula in the lemma below that will play crucial roles in several places of this paper; the proof is straightforward and we omit it.

\begin{lem}\label{lem:seq}
There exists an $L_{\bG,\ro}$--formula $\on{seq}(u,v,\gamma)$ such that 
\[\act_{\bG,\ro}(M,G)\models \on{seq}(U,V,g)\]
 for $U,V\in\ro(M)$ and $g\in G$ if and only if the following conditions hold
 for a unique $U'\in\pi_0(U)$:
\begin{enumerate}[(i)]
\item the set $U$ is dispersed;
\item we have that $V\cup g(V)\sse U$;
\item for all $\hat U\in\pi_0(U)$, the set $\hat U\cap V$ is nonempty and connected;
\item for all $\hat U\in \pi_0(U)\setminus\{U'\}$, the set $\hat U\cap g(V)$ is nonempty and connected;
\item\label{p:diff} we have that $U'\cap g(V)=\varnothing$;
\item\label{p:minimal} if a union $W$ of connected components of $U$ satisfies that $U'\sse W$
and that $g(V\cap W)\sse W$, then $W=U$.
\end{enumerate}
\end{lem}

In a situation as in Lemma~\ref{lem:seq}, we can enumerate the components of $U$ as
\[
\hat U_0=U',\hat U_1,\ldots\]
so that $g(V\cap \hat U_i)\sse \hat U_{i+1}$ for each $i\ge0$.
Furthermore, we have an injection
\[\sigma=\sigma_{U,V,g}\co\pi_0(U)\longrightarrow\pi_0(U)\setminus\{U'\}\]
sending $\hat U_i$ to $\hat U_{i+1}$ for each $i\in\omega$.
We also note that for each $i\in\omega$ there exists a uniformly definable function $\on{seq}_i(u,v,g)$ such that
\[
\on{seq}_i(U,V,g)=\hat U_i.\]

We can now establish the main result of this subsection.
\begin{lem}\label{lem:infcomp}
There exists a formula $\on{infcomp}(w)$ such that 
 \[\models \on{infcomp}(W)\quad\textrm{if and only if }W\text{ has infinitely many connected components}.\]
 \end{lem}
\begin{proof}
Let us define
\[\on{infcomp}(w)\equiv 
(\exists u,v,\gamma)[
\varnothing\ne u\sse w \wedge 
(\forall \hat w\in\pi_0(w))[\on{conn}(u\cap \hat w)]
\wedge
\on{seq}(u,v,\gamma)
].\]

In order to prove the forward direction, suppose we have $\on{seq}(U,V,g)$
for some nonempty $U\sse W$, such that each connected component $\hat W$ of $W$ satisfies $\on{conn}(U\cap \hat W)$. In particular, we have $|\pi_0(U)|\le|\pi_0(W)|$. The injection $\sigma_{U,V,g}$ 
above certifies that $\pi_0(U)$ is an infinite set.
Hence, $\pi_0(W)$ is infinite as well.

For the backward direction, suppose that $W$ has infinitely many components.
We will establish $\models \on{infcomp}(W)$ only in the case of $(M,G)\in \cM_{\vol}$, since the case  $(M,G)\in \cM$ is strictly easier.
We use an idea similar to the proof of Proposition~\ref{prop:full-support}.
We first find distinct components $\{\hat W_i\}_{i\in\omega}$ of $W$ such that some sequence $\{x_i\}_{i\in\omega}$ satisfying $x_i\in \hat W_i$ converges to some point $x^*\in M$. 
We consider a sufficiently small compact chart neighborhood $B$ of $x^*$, which still intersects infinitely many components of $W$. 
Let $n=\dim M$.
By the Oxtoby--Ulam theorem, we can
simply identify $B$ with $B^n(1)$ or $B^n(1)\cap\bH^n_+$ equipped with the Lebesgue measure. The point $x^*$ is then identified with the origin $O$.
By shrinking each $\hat W_i$ to $\hat U_i\sse \inte B$ and passing to a subsequence,
we can further require the following for all $i\ge0$.
\begin{itemize}
\item The open set $\hat U_i$ is an open Euclidean ball, converging to $x^*=O$;
\item We have 
$
\on{dist}(x^*,\hat U_{i+1})+\diam(\hat U_{i+1})<\on{dist}(x^*,\hat U_i)$.
\end{itemize}
We set  \[U:=\bigsqcup_i \hat U_i=\oplus_i \hat U_i\] and $U'=\hat U_0$.
We can find a disjoint collection of compact topological balls  $\{D_i\}$ such that $\inte D_i$ intersects both $\hat U_i$ and $\hat U_{i+1}$, and no other $\hat U_j$'s. 
Using the path--transitivity as in  Lemma~\ref{lem:transitive},
we can inductively find  
a \[g_i\in G[\inte D_i]\]
sending some  good ball $C_i\sse \hat U_i$ onto another  good ball inside $\hat U_{i+1}$.
We will set  \[V=\bigcup_{i\in\omega} \inte C_i.\]
By the uniform convergence theorem, the sequence 
\[\left\{\prod_{i=0}^k g_i\right\}_{k\ge0}\]
converges to a homeomorphism $g\in\Homeo_{0,\mu}(M)\le G$, which witnesses the properties that the formula $\on{infcomp}(U)$ requires.
\end{proof}

It follows immediately that we may also test whether a regular open set has finitely many components, and write \[\on{fincomp}(u)\equiv\neg\on{infcomp}(u).\]   

\subsubsection{Touching and containing the boundary}
By a \emph{collar (embedding)} of the boundary in a manifold $M$, we mean an embedding \[h\co \partial M\times [0,1)\longrightarrow M\] that extends the identity map \[\partial M\times\{0\}\longrightarrow \partial M;\]
we sometimes allow $h$ to be an embedding of $\partial M\times [0,1]$.
The image of a collar embedding is called a \emph{collar neighborhood}.
A fundamental result due to Brown~\cite[Theorem 2]{Brown1962} says that the boundary of a topological manifold admits a collar.
We now produce several formulae regarding the boundary of a given manifold.

\begin{lem}\label{lem:touch}
There exist $L_{\bG,\ro}$--formulae as follows:
\be
\item\label{p:touch-partial}
A formula $\on{touch}_\partial(u)$ such that
\[\models \on{touch}_\partial(U)
\quad\textrm{if and only if the closure of }U\text{ nontrivially intersects }\partial M.\]
\item\label{p:stab-partial}
A formula $\on{stab}_\partial(\gamma)$ such that 
\[\models \on{stab}_\partial(g)
\quad\textrm{if and only if }
g\text{ setwise stabilizes each boundary component of }M.\]
\ee\end{lem}
\begin{proof}
(\ref{p:touch-partial})
 Let us define the formula 
\[\on{finint}(u,w)\equiv(\exists u')[\on{fincomp}(u')\wedge(\forall \hat u\in\pi_{0}(u))[\hat u\cap w\ne\varnothing\rightarrow \hat u\in\pi_0(u')]].\]
It is clear from the formulation that \[\models\on{finint}(U,W)\] if and only if $U$ meets $W$ in finitely many components on $U$. 
We now set
\begin{align*}
\on{touch}_\partial(u)\equiv&(\exists u')[u'\subseteq u\wedge \on{infcomp}(u')\wedge(\forall w)
[\on{RCB}(w,M)\longrightarrow \on{finint}(u',w)]].
\end{align*}

Suppose that $\cl U\cap\partial M\neq\varnothing$. Choose a sequence of points $\{p_i\}_{i\in\omega}$ in $U$ converging to
a point in $\partial M$, and choose small open balls $U_i\ni p_i$ in $U$ with pairwise disjoint closures and with radii tending to zero.
Let $U'$ be the union of these balls. Now, if $W$ fails to
satisfy $\on{finint}(U',W)$, then $W$ must meet infinitely many of the balls $U_i$; thus $\cl W\cap\partial M\neq\varnothing$. In
particular, $\neg\on{RCB}(W,M)$.
	
Conversely, suppose that $\cl U\cap\partial M=\varnothing$, and let $U'\sse U$ 
have infinitely many components $\{U_i\}_{i\in\omega}$. As in Lemma~\ref{lem:infcomp}, by shrinking components of
 $U'$ and passing to a subsequence, we may assume that each $U_i$ is an open ball,
that the sequence has shrinking radii, and converges monotonically to the origin in an open chart in $\bR^n$. Moreover, the origin
in this chart lies in the interior of $M$, by assumption.
	
We may take $W$ to be a neighborhood of the origin in this chart, which then satisfies $\on{RCB}(W,M)$ and meets infinitely
many components of $U'$. Thus, $U'$ fails to witness $\on{touch}_\partial(U)$, and so $\on{touch}_\partial(U)$ does not hold.

(\ref{p:stab-partial})
Setting
\[\on{contain}_\partial(u)\equiv\neg\on{touch}_\partial(u^\perp),\]
we see that $\on{contain}_\partial(U)$ holds if and only if $\partial M\sse U$.
We now define \[\on{stab}_{\partial}(
\gamma)\equiv (\forall u,\hat u)[(\hat u\in\pi_0(u)\wedge\on{contain}_\partial(u)\wedge 
\on{touch}_{\partial}(\hat u))\rightarrow \hat u\cap\gamma(\hat u)\neq\varnothing].\]

We claim that $\on{stab}_\partial(g)$ holds for $g\in G$ if and only if $g$ setwise stabilizes each component of \[\partial M=\partial_1\sqcup\partial_2\sqcup\cdots\sqcup \partial_k.\]
For the forward direction, suppose we have $\on{\stab}_\partial(g)$.
By the aforementioned result of Brown, we can pick a closure--disjoint collection of collar neighborhoods $\{V_1,V_2,\ldots,V_k\}$ of the components
of $\partial M$.
Defining \[U:=\bigsqcup_{i=1}^k V_i,\] we see from the hypothesis that $g(V_i)\cap V_i\ne\varnothing$ for each $i$, which trivially implies $g(\partial_i)=\partial_i$.
The backward direction is clear after observing that the hypothesis of $\stab_\partial(g)$
simply says that $\hat U$ contains at least one boundary component.\end{proof}

\section{Interpretation of second-order arithmetic}\label{sec:arith}
The goal of this section is to prove that the group $G$
interprets second order arithmetic and analysis uniformly for $(M,G)\in\cM_{(\vol)}$,
establishing the case of $i=1$ in Proposition~\ref{prop:induct}.

\subsection{An example of an interpretation of first order arithmetic}\label{ss:arith1}
As a warm-up, let us interpret first order arithmetic 
\[
\on{Arith}_1=(\bN,+,\times,0,1)\]
in the structure $\act_{\bG,\ro}(M,G)$. 
For this, we consider the surjection
\[
\#\pi_0\co \{U\in\ro(M)\mid \on{fincomp}(U)\}\longrightarrow \bN\]
sending each $U$ to $\#\pi_0(U)$, namely the cardinality of $\pi_0(U)$.
The domain of this surjection is clearly definable,
and so is the fiber by the formula $\#_=(u,v)$ in Corollary~\ref{cor:basic-form}.
To complete an interpretation of $\on{Arith}_1$, it suffices to establish the following:
\begin{lem}\label{lem:arith1}
There exist $L_{\bG,\ro}$--formulae $\#_+$ and $\#_\times$ such that the following hold for all $U,V,W$ having finitely many connected components. 
\begin{enumerate}
    \item We have $\models\#_{+}(U,V,W)\quad \textrm{if and only if}\quad \#\pi_0(W)=\#\pi_0(U)+\#\pi_0(V)$.
     \item We have $\models\#_{\times}(U,V,W)\quad \textrm{if and only if}\quad \#\pi_0(W)=\#\pi_0(U)\cdot\#\pi_0(V)$.
\end{enumerate}
\end{lem}
\begin{proof}
Recall the meaning of the formula $\on{ccpartition}$ from 
Corollary~\ref{cor:basic-form}.  Let us make the following definitions.
\begin{align*}
\#_{+}(u,v,w)&\equiv(     
     \exists w_{0},w_{1})[\on{ccpartition}(w_0,w_1,w)
     \wedge \#_{=}(w_{0},u)\wedge
	 \#_{=}(w_{1},v)],\\
\#_{\times}(u,v,w)&\equiv(\exists w')[ (w'\sse u)\wedge\#_{=}(w,w')\wedge(\forall \hat{u}\in\pi_0(u))[\#_{=}(\hat{u}\cap w',v)]].
\end{align*}
It is straightforward to check that these formulae have the intended meanings.\end{proof}

\subsection{Our interpretation of second order arithmetic}\label{ss:second-arith}
We now describe an interpretation of second order arithmetic
\[\on{Arith}_2
=	(\bN,\PP (\bN),0,1,+,\times,\in),\]
which has two sorts, namely $\bN$ and $\PP(\bN)$.
In particular, we will have to be able to quantify over subsets of $\bN$. 

In order to achieve this, we will consider more restricted class of regular open sets $U$, the components of which admit a linear order as described by the formula $\on{seq}(U,V,g)$; see Section~\ref{ss:fin}.
In this linear order of $U$, the $k$--th component $\hat U_k$ will interpret the integer $k\in\omega$, and a union of the connected components $W$ will interpret a subset in a natural way. 
We will utilize Lemma~\ref{lem:arith1}, but not the actual interpretation itself from the previous subsection.

To be more concrete, let us first note the following. 
\begin{lem}\label{lem:seq-all}
There exists a uniformly defined function $\on{seq}_\restriction(u,v,\gamma,\hat u)$ 
such that if  
 \[\models \on{seq}(U,V,g)\wedge \hat U\in\pi_0(U),\]
then for the unique $k\in\omega$ satisfying
$\hat U=\on{seq}_k(U,V,g)$, we have that
\[ \on{seq}_\restriction(U,V,g,\hat U)
= \oplus_{0\le i\le k}\on{seq}_i(U,V,g).\]
\end{lem}
\begin{proof}
It is routine to check that the following has the intended meaning:
\begin{align*}
(w=\on{seq}_\restriction(u,v,\gamma,\hat u))
\equiv&
\on{ucc}(w,u)\wedge
(\on{seq}_0(u,v,\gamma)\oplus \hat u)\sse w\wedge\\
&\gamma(v\cap \hat u)\cap w=\varnothing\wedge
\gamma^{-1}(w\cap \gamma(v))\sse w.\qedhere
\end{align*}
\end{proof}

Let us consider the set
\[
X_1:=\{(U,V,g,\hat U)\mid \on{seq}(U,V,g)\wedge \hat U\in\pi_0(U)]\},\]
which is definable in $\act_{\bG,\ro}(M,G)$ uniformly for $(M,G)\in\cM_{(\vol)}$.
We have a surjection \[\rho_1:=\#\pi_0\circ\on{seq}_\restriction\co X_1 \longrightarrow \bN.\]
This surjection satisfies
\[k=\rho_1(U,V,g,\hat U)\quad\text{if and only if}\quad \hat U=\on{seq}_k(U,V,g).\]
 The fiber of $\rho_1$ is
\[
\{(\underline{y},\underline{z})\in X_1\times X_1\mid 
\#_=(\on{seq}_\restriction(\underline{y}),\on{seq}_\restriction(\underline{z}))\},\]
and hence uniformly definable. It is trivial to check that $\rho_1$ produces a uniform interpretation of $\on{Arith}_1$ to $\act_{\bG,\ro}(M,G)$.
For instance, we have
\[
\rho_1(\underline{y})+\rho_1(\underline{y}')=\rho_1(\underline{y}'')
\]
if and only if
\[\models\#_+(\on{seq}_\restriction(\underline{y}),\on{seq}_\restriction(\underline{y}'),\on{seq}_\restriction(\underline{y}'')).\]
After this interpretation of $\on{Arith}_1$, the symbol $\#$ has an intended meaning as a function from $\ro(M)$ to $\bN$.
We  have uniformly defined functions  $\on{seqcomp}(u,v,\gamma,\alpha)$
and 
$\on{seqcomp}_\restriction(u,v,\gamma,\alpha)$
satisfying
\begin{align*}
\on{seqcomp}(U,V,g,k)&=\on{seq}_k(U,V,g),\\
\on{seqcomp}_\restriction(U,V,g,k)&=\oplus_{i=1}^k\on{seqcomp}(U,V,g,i).
\end{align*}

Similarly, we consider another uniformly definable set
\[
X_1':=\{(U,V,g,W)\mid \on{seq}(U,V,g)\wedge \on{ucc}(W,U)]\}.\]
We have a surjection
\[
\rho_1'\co 
X_1'\longrightarrow \PP(\bN)\]
defined by the condition
\[
\rho_1'(U,V,g,W):=
\{\rho_1(U,V,g,\hat W)\mid \hat W\in\pi_0(W)\}.\]
Since the fiber of $\rho_1$ is uniformly definable, so is that of $\rho_1'$. 
We will introduce the function symbol $\PP_\#$ in $L^2_{\act}$ interpreted as $\rho_1'$.

Finally,  we have 
\[\rho_1(U_1,V_1,g_1,\hat U)\in\rho_1'(U_2,V_2,g_2,W_2)\]
if and only if
\[
\#\pi_0\on{seq}_\restriction(U_1,V_1,g_1,\hat U)
=\#\pi_0\on{seq}_\restriction(U_2,V_2,g_2,\hat W)\]
for some $\hat W\in\pi_0(W_2)$.
Hence, the pair of surjections $(\rho_1, \rho_1')$ produces the desired interpretation of the two--sorted structure $\on{Arith}_2$. We note that the order relation symbol $<$, the successor symbol $S$,
and the inclusion symbol $\sse$ are naturally interpreted as a consequence.

\subsection{Analysis}
The interpretation of $\bR$ is now standard. From $\bN$, we interpret $\bZ$, together with addition, multiplication, and order, by imposing a suitable definable equivalence relation on a suitable definable subset of $\bN^2$. We similarly interpret $\bQ$ by imposing a suitable definable equivalence relation on a definable subset of $\bZ\times \bZ$.

We define $\bR$ together with addition, multiplication, and order via Dedekind cuts of $\bQ$; all this is interpretable because of our access to $\PP(\bN)$. Finally, we have canonical identifications of \[\bN\subseteq\bZ\subseteq\bQ\subseteq\bR,\] wherein we set $=$ to be the relation identifying natural numbers with their images under this sequence of inclusions. In the sequel, we will simply talk about natural numbers, integers, or rationals as elements of $\bR$ without further comment. 
We further may assume to have $\bR^k$ in the universe of the structure for all $k\in\omega$.

In order to justify the introduction of the sort symbol $\cont_{k,\ell}$ in the structure,
let us first note that each function in $C(\bR,\bR)$ is uniquely determined by its restriction on $\bQ$. Since
\[
|\bR^\bQ|=(2^\omega)^\omega=2^\omega=|\bR|,\]
we have an interpretation of $C(\bR,\bR)$ by $\bR$, and hence, that of \[C(\bR^k,\bR^\ell).\]
This latter set is the domain of $\cont_{k,\ell}$, and the function symbols
\[
\on{appl}(\chi,\rho)=\sigma, \quad \on{norm}(\chi)=\rho\]
are interpreted accordingly. In practice, we write
\[
f(r)=s,\quad \|f\|=r\]
for the above formulae.
The expanded language containing $\act_{\bG,\ro}$ structure, second order arithmetic, and analysis will be written $\act^2=\act_{\bG,\ro,\bR}$.
This establishes the uniform interpretability of $\act_{\bG,\ro,\bR}(M,G)$ to $\act_{\bG,\ro}(M,G)$, namely Proposition~\ref{prop:induct} for the case $i=1$.

\section{Interpretation of points}\label{sec:points}
We now wish to be able to talk about points of $M$ more directly, and prove Proposition~\ref{prop:induct} for the case $i=2$. This will complete the proof of Theorem~\ref{thm:interpret}.

Rubin~\cite{Rubin1989} accesses points in a space with a locally dense action via a certain collection of ultrafilters consisting of regular open sets; in his approach, the intersection of the closures of all the open sets in each ultrafilter corresponds to a single point of the space. 
We cannot follow this approach directly, as we need to stay within the first order theory of groups and Boolean algebras. Instead, we consider a certain collection of regular open sets such that the components in each of those open sets converge to a single point of the manifold. We continue to make the standing assumption that $(M,G)\in\cM_{(\vol)}$ with $\dim M>1$, unless stated otherwise.

\subsection{Encoding  points of a manifold}\label{ss:points}
Using the $L_{\bG,\ro,\bR}$--formulae introduced in the preceding sections, 
we define the following new formulae:
\begin{align*}
\on{cof}(w,u)
\equiv&\on{infcomp}(w)\wedge w\sse u \wedge
(\forall \hat u\in\pi_0(u))[\on{conn}(w\cap\hat u)]
,\\ 
\on{cofcontain}(w,u)\equiv&(\exists w')[w'\sse w \wedge \on{cof}(w',u)],\\
\on{cofmove}(\gamma,u_0,u_1)
\equiv&
(\exists w)[\on{cof}(w,u_0)\wedge
\on{cof}(\gamma(w),u_1)\wedge
(\forall \hat w\in\pi_0(w))[\on{RCB}(\hat w,u_0)]]
\end{align*}

Note that when $\models\on{cofmove}(g,U_0,U_1)$,
we can find some $W$ whose connected components can be written as \[W=\bigsqcup_{i\in\omega} W_i,\] 
with the property that each $W_i$ is contained in some relatively compact ball inside $U_0$;
moreover, no two components of $W$ belong to the same component of $U_0$,
and similarly for $g(W)$ and $U_1$.

We consider the definable set
\[
\lst_G(U):=\bigcup \{G[W^\perp]\mid W\in\ro(M)\text{ and }\on{cof}(W,U)\},\]
which we call as the \emph{limit stabilizer} of $U$ in $G$.
Intuitively, each element of this set fixes some open set that is arbitrarily close to a certain limit point of the components of $U$.
We will write $\gamma\in\lst(u)$ for the formula corresponding to $g\in \lst_G(U)$.

\begin{rem}
One can rephrase Rubin's interpretation of points in second order logic~\cite{Rubin1989} as follows, as summarized in~\cite[Theorem 3.6.17]{KK2021book}. Rubin allowed certain collections (called, \emph{good ultrafilters}) of regular open sets to interpret a single point in the space, by taking the intersection of the closures of those open sets. He then proved that two good ultrafilters $P$ and $Q$ interpret different points $p$ and $q$
if and only if the group
\[G\{Q^\perp\}:=\bigcup\{ G[W^\perp]\mid W\in Q\}\]
acts \emph{sufficiently transitively}, in the sense that
for some $U\in P$, 
every $V\in \ro(M)$ satisfying $\varnothing\ne V\sse U$
is an element of the set
\[
G\{Q^\perp\}(P).\]
In our approach, we will utilize the sufficient transitivity of the limit stabilizer characterized in terms of the formula $\on{cofmove}(\gamma,w_0,w_1)$.
\end{rem}

Consider the set $X_2:=\act^2(M,G)(\phi_{\on{dom}}^2)$,
defined by the following formula:
\begin{align*}
\phi_{\on{dom}}^2(u,v,\gamma)\equiv&
\on{seq}(u,v,\gamma)\wedge
(\forall w_0,w_1)[\on{cof}(w_0,u)\wedge\on{cof}(w_1,u)\longrightarrow\\
&(\exists\delta\in\lst(w_0))[\on{cofmove}(\delta,w_0,w_1)]].\end{align*}

The following lemma furnishes an interpretation of the points.

\begin{lem}\label{lem:rho3}
For each $(U,V,g)\in X_2$
and for an arbitrary sequence $\{x_i\}_{i\in\omega}$ satisfying \[x_i\in \on{seqcomp}(U,V,g,i)\] for all $i\in\omega$, the limit
\[
\rho_2(U,V,g):=\lim_{i\to\infty} x_i\]
exists in $M$, and is independent of the choice of $\{x_i\}_{i\in\omega}$.
Moreover, the following conclusions hold:
\begin{enumerate}
\item\label{p:surj}  The map
$\rho_2\co X_2\longrightarrow M$ is surjective.
\item\label{p:eq} We have 
\[
\rho_2(U_0,V_0,g_0)= \rho_2(U_1,V_1,g_1)\]
if and only if
some $g\in\lst_G(U_0)$ satisfies
 \[\on{cofmove}(g,U_0,U_1).\]
\item\label{p:action} We have 
\[
h(\rho_2(U,V,g))=\rho_2(U',V',g')\]
if and only if
\[\rho_2(h(U),h(V),hgh^{-1})=\rho_2(U',V',g').\]
\item\label{p:membership} We have $\rho_2(U,V,g)\not\in W$ 
if and only if
some $(U',V',g')\in X_2$ satisfies
\[
U'\cap W=\varnothing\wedge\left(
\rho_2(U,V,g)= \rho_2(U',V',g')\right).\]
\item\label{p:bdry} We have $\rho_2(U,V,g)\in\inte M$
 if and only if 
 there exists some $W$ such that $\on{RCB}(W,M)$ and such that 
 $\on{cofcontain}(W,U)$.
\end{enumerate}
\end{lem}

\begin{proof}
Let $(U,V,g)\in X_2$, and let 
\[\{x_i\in\on{seqcomp}(U,V,g,i)\}_{i\in\omega}\] be a sequence.
In particular, we have $x_i\in\inte M$.
Suppose two subsequences
\[\{y_{0,j}\}_{j\in\omega},\{y_{1,j}\}_{j\in\omega}\subseteq \{x_i\}_{i\in\omega}\] converge to two distinct points $y_0$ and $y_1$.
For $i=0$ and $i=1$, we let $W_i$ be the union of sufficiently small good open balls $W_{i,j}$ centered at $y_{i,j}$.
In particular, we may assume that $\on{cof}(W_i,U)$,
and that \[\lim_j W_{i,j}=\{y_i\}\]
in the Hausdorff sense.
By hypothesis, we have some $h\in \lst(W_0)$ such that 
\[\models\on{cofmove}(h,W_0,W_1).\] Since $h$ fixes points  arbitrarily close to $y_0$, we have $h(y_0)=y_0$. It follows that
\[
y_0=h(y_0)=\lim h(y_{0,j})= y_1.\]
This proves the existence of the claimed limit. 
The same argument also implies the independence of the limit from the choice of $\{x_i\}_{i\in\omega}$,
and also the backward direction of part~(\ref{p:eq}). The surjectivity of $\rho_2$ in part (\ref{p:surj}) is clear, after choosing $U$ to be a suitable sequence of good open balls converging to a given point in the Hausdorff sense.

We now verify the forward direction of part~(\ref{p:eq}).
By hypothesis, we can find two sequences $\{x_{0,j}\}_{i\in\omega}$ and $\{x_{1,j}\}_{i\in\omega}$ such that 
\[
x_{i,j}\in \on{seqcomp}(U_i,V_i,g_i,j).\]
As in the proof of Lemma~\ref{lem:infcomp}, we can find a disjoint collection of good balls $\{D_i\}$ of decreasing sizes such that each $D_i$ contains $x_{0,j}$ and $x_{1,j}$, 
after passing to a subsequence if necessary.
By the uniform convergence theorem, we have some $h\in G$ 
such that $h(x_{0,j})=x_{1,j}$ for all $j$,
and such that
$h$ pointwise fixes some nonempty open set inside 
\[\on{seqcomp}(U_0,V_0,g_0,j)\cap D_j^\perp.\]
In particular, we have that $h\in\lst_G(U_0)$ and that $\on{cofmove}(h,U_0,U_1)$, as claimed.
The remaining parts of the lemma are straightforward to check.    
\end{proof}

In part~(\ref{p:eq}) of the lemma, we see that
the relation
\[\rho_2(U,V,g)=\rho_2(U',V',g')\]
is first order expressible;
hence, we deduce that the functional relation $g(p)=q$ and the membership relation $p\in W$ in parts~(\ref{p:action}) and~(\ref{p:membership}) are interpretable for $p,q\in M$, $g\in M$ and $W\in\ro(M)$. Part (\ref{p:bdry}) of the lemma separates out the interior points.

Direct access to points allows us to make  direct reference to set theoretic operations. For instance, we can define $\union(u,v,w)$ by
    \[\union(u,v,w)\equiv(\forall\pi)[(\pi\in u\vee\pi\in v)\leftrightarrow \pi\in w].\] Clearly, $\union(U,V,W)$ for regular open sets $\{U,V,W\}$
    if and only if $W=U\cup V$. Henceforth, we will include the usual set-theoretic union symbol in the language such as $\cup, \cap$ and $\setminus$.    
We are also able now to talk directly about the closure $\cl U$ of a regular open set $U$, both in $M$ and in $V$ for arbitrary $U\subseteq V$; for this, it suffices to note that $p\in \cl U$ if and only if $p\not\in U^\perp$.

\subsection{Encoding discrete sets of points in a manifold}\label{ss:discrete}
We now interpret the set 
\[\PP^{\on{disc}}(\inte M):=\{A\sse \inte M\mid A\text{ is discrete}\}.\]
In particular, every finite subset of $\inte M$ belongs to $\PP^{\on{disc}}(\inte M)$.

We recall from Lemma~\ref{lem:touch} the formula $\on{finint}(u,w)$. 
We first let $X_2'$ be the set of quadruples $(U,V,g,W)$ defined by the following formula:
\begin{align*}
\psi_{\on{dom}}^2(u,v,\gamma,w)\equiv&
\on{dispersed}(w)\wedge (u\oplus\suppe \gamma\sse w)\wedge\\
&\forall \hat w\in\pi_0(w)[\phi^2_{\on{dom}}(u\cap \hat w,v\cap \hat w,\gamma)
\wedge\on{RCB}((u\oplus\suppe \gamma)\cap \hat w,\hat w)].\end{align*}


For such a quadruple, we set
\[
\rho_2'(U,V,g,W):=\{\rho_2(U\cap\hat W,V\cap \hat W,g)\mid \hat W\in\pi_0(W)\}.\]
It is routine to check that this map defines a surjection
\[
\rho_2'\co X_2'\longrightarrow\PP^{\on{disc}}(\inte M)\]
 with a definable fiber.
Namely, we have \[
\rho_2'(U_0,V_0,g_0,W_0)\ne \rho_2'(U_1,V_1,g_1,W_1)\]
if and only if there exists some regular open sets $W',W''$
satisfying
that
\[\on{RCB}(W',W'')\]
and that
\[
\neg\on{finint}(U_i,W')\wedge\on{finint}(U_{1-i},W'')\]
for some $i\in\{0,1\}$.

We  interpret the membership between a point and a set;
namely, we have
\[\rho_2(U,V,g)\in\rho_2'(U',V',g',W')\]
if and only if there exists some $W''$ satisfying $\on{RCB}(W'',W')$
and 
\[
\on{cofcontain}(W'',U).\]
We also interpret the group action 
\[
h(\rho_2'(U,V,g,W))=\rho_2'(U',V',g',W'))\]
as
\[\rho_2'(h(U),h(V),hgh^{-1},h(W)))=\rho_2'(U',V',g',W'))\]
Finally, the set \[\rho_2'(U,V,g,W)\in\PP^{\on{disc}}(\inte M)\] has finite cardinality if and only if $W$ has finitely many connected components. In this case, the cardinality function $\#$ for $T\in\PP^{\on{disc}}(\inte M)$ is clearly definable by  \[\#(\rho_2'(U,V,g,W))=\#\pi_0(W).\]
We omit the details, which are very similar to those in Section~\ref{ss:points}.
We denote by $\bM^{\on{disc-int}}$ the sort symbol for sets belong to $\PP^{\on{disc}}(\inte M)$.

\subsection{Interpreting exponentiation}
We now interpret the map
\[G\times\bZ\times M\longrightarrow M, \quad (g,k,p) \mapsto g^k\cdot p,\]
so that the exponentiation map \[\exp:G\times\bZ\longrightarrow G\] is definable.
Note that $g^k(p)=p'$ holds with $k\in\omega$ if and only if 
we can write $k=mq+r$ for some integers $0\le r<m$ and $q$
such that 
we have a period--$m$ orbit
\[p,g(p),\ldots,g^m(p)=p,\]
and a sequence of distinct points
\[
p,g(p),\ldots,g^r(p)=p'.\]
Let us now define formulae $\on{exp}_{\on{cyc}}$ and $\on{exp}_{\on{lin}}$, which will express the existences of a periodic orbit and of a sequence without repetitions, respectively. More precisely, we set
    \begin{align*}
    \exp_{\on{cyc}}(\gamma,\alpha,\pi)\equiv&
    (\alpha=0)\vee
    (\exists\tau\in\bM^{\on{disc-int}})
    [
    \#\tau=\alpha\wedge (\pi\in\tau=\gamma(\tau) )\\
    &\wedge\neg(\exists\tau'\in \bM^{\on{disc-int}})[\varnothing\ne\tau'\subsetneq\tau\wedge\gamma\cdot\tau'=\tau']],\\
    \exp_{\on{lin}}(\gamma,\alpha,\pi,\pi')\equiv&
    (\exists\tau\in \bM^{\on{disc-int}})
    [
   \#\tau=\alpha+1\wedge \{\pi,\pi'\}\sse\tau\wedge
    \gamma(\tau\setminus\{\pi'\})=\tau\setminus\{\pi\}\\
    	&\wedge\neg(\exists\tau'\in \bM^{\on{disc-int}})[\varnothing\ne\tau'\subsetneq\tau\wedge\gamma\cdot\tau'=\tau']].
    \end{align*}
    

We see that $\exp(g,k)\cdot p=p'$ with $k\ge0$ if and only if the tuple $(g,k,p,p')$ satisfies the formula 
\[
    \exp(\gamma,\alpha,\pi,\pi')\equiv
    (\exists \alpha',\beta_1,\beta_2)
    [\alpha=\beta_2\alpha'+\beta_1\wedge \exp_{\on{cyc}}(\gamma,\beta_2,\pi)\wedge
    \exp_{\on{lin}}(\gamma,\beta_1,\pi,\pi')].\]
It is then trivial to extend the definition for the case $k<0$, establishing the definability of the exponentiation function.

\subsection{The $\agape$ structure}
We now define our ultimate structure 
\[\act^3(M,G)=\act_{\bG,\ro,\bR,\bM}(M,G)=\agape(M,G)\]
as the extension of $\act^2(M,G)=\act_{\bG,\ro,\bR}(M,G)$
by including the points in $M$
and adding the relations 
\[g(p)=q,\quad p\in W\]
for $g\in G$, $p,q\in M$ and $W\in\ro(M)$.
We are then justified to use expressions such as
\[
p\in\inte M,\quad p\in\partial M,\quad p\in \cl U,g^n=h,\quad \fix g=\cl U,\quad U\cup V = W\]
for points $p$, regular open sets $U,V,W$, group elements $g,h\in G$
and integer $n\in\bZ$ within $\agape(M,G)$.

\section{Balls with definable parametrizations}\label{sec:balls}
From this point on, we work in the $\agape$ language
$L_{\agape}=L_{\bG,\ro,\bR,\bM}$, containing second order arithmetic and points. 
The underlying structure will be $\agape(M,G)$; 
recall our further standing assumption that $\dim M>1$.
We will use the notation $I^n=[0,1]^n$ and \[Q^n(r):=[-r,r]^n.\]
The main objective of this section is to interpret the dimension and collared balls inside of a manifold, as described in the following two theorems:

\begin{thm}\label{thm:dim}
For each $n\ge 2$, there exists a formula $\dim_n$ such that $\models\dim_n$ if and only if $M$ is an $n$--manifold.
\end{thm}

\begin{thm}\label{thm:output-balls}
For each $n\ge 2$, there exist formulae 
\[\on{flows}_n(u,\gamma,\pi),\on{Param}_n(u,\gamma,\pi,\rho,\pi')\]
such that the following hold for all $(M,G)\in\cM_{(\vol)}$ with $n=\dim M$.
\begin{enumerate}
\item\label{p:flows1}
Let $U\in\ro(M)$, $\underbar g\in \cH^n$ and $p\in M$.
If    	\[    \models\on{flows}_n(U,\underbar g,p)    	\] 
then there exists a unique homeomorphism 
\[
\Psi=\Psi[U,\underline{g},p]\co I^n\longrightarrow \cl U\]
the graph $\Gamma$ of which satisfies
\[\Gamma=\{(r,q)\in I^n\times M\co\agape(M,G) \models	 \param_n(U,\underbar g,p,r,q)\},\]
and also  $(0,p)\in\Gamma$.
\item\label{p:flows2} Let $U$ and $V$ be good open balls inside $\inte M$ such that $\cl U\sse V$; if $(M,G)\in\cM_{\vol}$, we further assume that $\vol(U)/\vol(V)$ is sufficiently small compared to some positive number determined by $n$. Then we have \[ \models(\exists \underline{\gamma}\exists \pi)\,\on{flows}_{n}(U,\underline{\gamma},\pi).\] 
\end{enumerate}
  \end{thm}

In Section~\ref{sec:homeo}, we will modify the definition of $\Psi[U,\underline{g},p]$ so that the domain is $Q^n(2)$, instead of $I^n$.
We emphasize again that the above formulae for $\cM$ and $\cM_{\vol}$ may differ; for instance, the abbreviated sentence $\dim_n$ could be more precisely denoted by $\dim_n$ and $\dim_n^{\vol}$ separately depending on the context.

  \subsection{Detecting the dimension of a manifold}
We prove Theorem~\ref{thm:dim} by interpreting a sufficient amount of dimension theory.
For a topological space $X$,
the \emph{order} of a finite open cover $\UU$  is defined as the number
\[
\sup_{x\in X}\; \abs*{\{U\in\mathcal{U}\mid x\in U\}}.\]
Though in classical literature one considers general open covers,
it is sufficient (especially in our situation) to consider finite
covers only; cf.~\cite{Coornaert2005,Edgar2008}.

We say the \emph{topological dimension} of $X$ is at most $n$, and write $\dim X\le n$,
if every finite open cover of $X$ is refined by an open cover with order at most $n+1$.
The topological dimension $\dim X$ is defined to be $n$, if $\dim X\le n$ holds but $\dim X\le n-1$ does not.
A topological $n$--manifold has the topological dimension $n$. 

A collection of open sets $\VV=\{V_i\}_{i\in \II}$ is said to \emph{shrink} to another collection $\WW=\{W_i\}_{i\in \II}$ if $W_i\sse V_i$ holds for each $i$ in the index set $\II$. Let us note the following well-known facts.

\begin{lem}\label{lem:dim-facts}
\begin{enumerate}
    \item\label{p:lebesgue} (Lebesgue's Covering Theorem~\cite[Theorem IV.2]{HW1941}) If $\UU$ is a finite open cover of $I^n$ such that no element of $\UU$ intersects an opposite pair of codimension one faces,
    then $\UU$ cannot be refined by an open cover of order at most $n$. 
    \item\label{p:monotone} (Čech~\cite{Cech1933}) If $X$ is a metrizable space and if $Y\sse X$, then $\dim Y\le \dim X$.
    \item\label{p:ostrand} (Ostrand's Theorem~\cite[Theorem 3]{Ostrand1971}) If $\UU=\{U_i\}_{i\in \II}$ is a locally finite open cover of a normal space $X$ satisfying $\dim X\le n$, then for each $j=0,\ldots,n$, the cover $\UU$ shrinks to some pairwise disjoint collection $\VV^j=\{V_i^j\}_{i\in \II}$ of open sets such that
    the collection $\bigcup_j \VV^j$ is a cover.
\end{enumerate}
\end{lem}


We can now give a characterization of manifold dimension.

\begin{lem}\label{lem:dim}
For each positive integer $n$ and for each compact manifold $M$, the following two conditions are equivalent.
\begin{enumerate}[(A)]
    \item\label{p:dimA} The dimension of $M$ is at most $n$;
    \item\label{p:dimB} 
Let $W$ be a regular open set in $M$.
    If   
    \[
    \UU=\{U_i\co i=1,2,\ldots,2^{n+1}\}\] 
    is a regular open cover of $\cl W$, then there exists a pairwise disjoint collection
    \[\VV^j=\{V^j_i\co i=1,2,\ldots, 2^{n+1}\}\]
    of regular open sets for each $j\in\{0,1,\ldots,n\}$ such that $\UU$ shrinks to each $\VV^j$, and such that $\bigcup_j \VV^j$ is a cover of $\cl W$.
\end{enumerate}
\end{lem}
\begin{proof}
Suppose we have $\dim M\le n$,
and assume the hypothesis of part (\ref{p:dimB}).
We see from Lemma~\ref{lem:dim-facts} (\ref{p:monotone}) that $\dim \cl W\le n$.
Part (\ref{p:ostrand}) of the same lemma implies that
$\UU$ shrinks to a pairwise disjoint collection of (not necessarily regular) open sets 
\[\WW^j=\{W_i^j\}_{i=1,\ldots,2^{n+1}}\] for each $j\in\{0,1,\ldots,n\}$ with the property that
$\bigcup_j \WW^j$ is a cover of the normal space $\cl W$.
By Lemma~\ref{lem:refine}, there exists a regular open cover
\[\VV:=\{V^j_i\}_{i,j}\]
of $\cl W$ satisfying \[\cl V^j_i\sse W^j_i\sse U_i\]
for all $i$ and $j$. This implies the conclusion of (B).

Conversely, suppose we have condition (B) and assume for contradiction that $m:=\dim M>n$. We first note the following:

\begin{claim*}
    The unit $m$--cube $[0,1]^m$ admits a finite regular open cover of cardinality $2^{n+1}$ that cannot be refined by another open cover with order at most $n+1$. 
\end{claim*}
Let $C$ denote the unit cube $[0,1]^{n+1}$ in $\bR^{n+1}$, which is embedded in $\bR^m$ as the subset with the last $m-n-1$ coordinates being zero. 
For each vertex $v\in C^{(0)}$, let us consider the translated open cube
\[U_v:=v+(-1,1)^{n+1}\sse \bR^{n+1}.\]
We then have a regular open cover
\[
\UU:=\{ U_v\co v\in C^{(0)}\}\]
of $C$ with cardinality $2^{n+1}$. Note that each open cube 
$U_v$ does not intersect an opposite pair of codimension one faces of $C$. 
By taking the Cartesian product $U'_v$ of each $U_v$ with $(-1,2)^{m-n-1}$, we obtain a finite regular open cover 
\[\UU'=\{U'_v\mid v\in C^{(0)}\}\] of $[0,1]^m$. 
If $\UU'$ is refined by another finite open cover $\VV$ of $[0,1]^m$ with order at most $n+1$, then the intersection of the elements in $\VV$ with $\bR^{n+1}\sse\bR^m$ is a finite open cover of $C=[0,1]^{n+1}$ with order at most $n+1$. 
This violates Lebesgue's Covering Theorem
(Lemma~\ref{lem:dim-facts}), and
the claim is thus proved.

Let us now consider a good ball $Q$ in $M$, which comes with an embedding 
\[
\phi\co \bR^m\longrightarrow M\]
satisfying $\phi[0,1]^m=\cl Q$.
By applying the above claim, we obtain a finite regular open cover of $\cl Q$ that cannot be refined by a finite open cover with order at most $n+1$. This contradicts condition (\ref{p:dimB}), which we have assumed. 
\end{proof}

Note that the cardinalities of covers $\UU$ and $\bigcup_j \VV^j$ in condition (\ref{p:dimB}) of the above lemma are explicitly bounded above by $2^{n+1}$ and $(n+1)2^{n+1}$, respectively. 
Note also that conditions such as
\[
\cl W\sse U_1\cup \cdots \cup U_{2^{n+1}}\]
are expressible in the $\agape$ language. 
It is therefore clear that  condition (\ref{p:dimB})
is expressible in this language, for each fixed positive integer $n$.
As a consequence, we obtain Theorem~\ref{thm:dim}.

\subsection{Parametrizing balls in \texorpdfstring{$M$}{M} in dimension two and higher}
For the proof of Theorem~\ref{thm:output-balls}, let us consider the quotient map
 \[\pr:\bR\longrightarrow \bR/\sq\bZ\] defined by
 \[x\mapsto [x]:=x+\sq\Z.\]
The image of $\bZ$ is dense in the circle $\bR/\sq\bZ$, equipped with the natural cyclic order.
The expression $\sqrt{2}$ will be regarded as a (definable) constant symbol in $L_{\agape}$.
We have chosen this value for concreteness, but
 for our purpose
 we could use an arbitrary irrational number that is definable without parameters
in arithmetic. 
There exists a definable function $\on{ang}(\rho_1,\rho_2)$ satifying
\[
r=\on{ang}(r_1,r_2)\]
if and only if the (unsigned) angular metric between $[r_1]$ and $[r_2]$ is $r\in[0,\sqrt{2})$.

Let us also define an $L_{\agape}$ formula 
\[\on{fcov}(u,v_0,\ldots,v_n)\equiv (\cl u)\subseteq\bigcup_{i=0}^n v_i\wedge\bigwedge_{i=0}^n\on{fincomp}(v_{i}).
\] We also use the formula
\[
\on{clshrink}(v_0,\ldots,v_n,v_0',\ldots,v_n')\equiv\bigwedge_{i=0}^n \cl v_i'\sse v_i.\]
We will equip $M$ with a compatible metric $d$, and denote by $d_\infty$ the induced uniform metric on the homeomorphism group. We have the following characterization of uniform convergence:
\begin{lem}\label{lem:unif-conv}
Let $U$ be a regular open set in $M$  such that $\cl U\sse \inte M$, and let 
\[
F_1\supseteq F_2\supseteq\cdots\]
be a sequence of subsets of $\Homeo(M)$ such that each $f\in F_1$ setwise stabilizes $U$.
Then the following two conditions are equivalent.
\begin{enumerate}[(A)]
    \item 
We have 
\[
\lim_{i\to\infty} \sup\{d_\infty(f\restriction_U,\Id\restriction_U)\mid f\in F_i\}=0.\]
\item Suppose we have two tuples of regular open sets 
    \[\underline{V}=(V_0,\ldots,V_n),\quad \underline{V}'=(V_0',\ldots,V_n')\] such that
    \[
    \on{fcov}(U,\underline{V})\wedge\on{fcov}(U,\underline{V}')\wedge \on{clshrink}(\underline{V},\underline{V}').\]
    Then there exist some $i\in\omega$ such that
    whenever a pair $(\hat V',\hat V)$ belongs to 
\[A:=\left\{(\hat V',\hat V)\in\bigcup_{j=0}^n \left(\pi_0(V_j')\times\pi_0(V_j)\right)
\middle\vert\; \hat V'\sse \hat V\right\},\]
each  $f\in F_i$ satisfies  \[f(\hat V'\cap \cl U)\sse \hat V.\]
\end{enumerate}
\end{lem}
\begin{proof}
Let us assume part (A), and also the hypotheses of (B). We set
\[
\epsilon_0:=\inf\left\{ d(\cl \hat V', M\setminus\hat V)\mid (\hat V',\hat V)\in A\right\},\]
which is positive since $A$ is finite.
Choosing $i$ so that 
\[
d_\infty(f\restriction_U,\Id\restriction_U)<\epsilon_0\]
for all $f\in F_i$, we obtain the conclusion.

Conversely, we assume the condition (B) and pick an arbitrary $\epsilon>0$.
Let $\UU$ be a finite cover of $\cl U$ by regular open sets with radius less than $\epsilon$.
Applying Lemma~\ref{lem:dim} (after replacing the number $2^{n+1}$ in the lemma by the size of $\UU$), we obtain a tuple of regular open sets
\[\underline{V}=(V_0,\ldots,V_n)\]
 such that every connected component of each $V_j$ has diameter at most $2\epsilon$,
and such that $\on{fcov}(U,\underline{V})$ holds.
By Lemma~\ref{lem:refine} and by compactness of $\cl U$, 
we obtain 
\[\underline{V}'=(V_0',\ldots,V_n')\] such that
\[\on{fcov}(U,\underline{V}')\wedge \on{clshrink}(\underline{V},\underline{V}').\]
Pick $i\in\omega$ as given by the condition (B),
and let $f\in F_i$ and $x\in \cl U$ be arbitrary. 
Since there exists some $(\hat V',\hat V)\in A$ 
such that $x\in \hat V'$, 
we see that
\[
d(x,f(x))\le \diam \hat V\le 2\epsilon.\]
This implies that $d_\infty(f\restriction_U, \Id\restriction_U)\le 2\epsilon$
and that condition (A) holds.
\end{proof}

We now interpret non-integral powers of group elements, in the following sense:
    \begin{lem}\label{lem:flow}There exist formulae
    \[\conv(u,\gamma,\rho,\delta), \quad\on{flow}(u,\gamma)\]
     such that the following hold
     for each $(M,G)\in\cM_{(vol)}$.
    \begin{enumerate}
    \item 
    For group elements $\{g,h\}\sse G$, a regular open set $U\in\ro(M)$, and
    a real number $r\in\bR$
     satisfying $\cl U\sse \inte M$ and \[g(U)=U=h(U),\] 
    we have \[\models \conv(U,g,r,h)\] if and only if  
      \[\lim_{\delta\to+0}\sup\left\{d_\infty(g^s\restriction_U,h\restriction_U)\mid s\in\bZ\text{ and }\on{ang}(s,r)<\delta\right\}=0.\]
\item 
    For  
    $g\in G$ and $U\in\ro(M)$
    satisfying $\cl U\sse \inte M$ and $g(U)=U$, 
    we have \[\models \on{flow}(U,g)\] if and only if
    there exists a unique topological flow
\[
\Phi=\Phi_{U,g}\co \bR/\sq\bZ\times U\longrightarrow U\]
such that, with the notation $\Phi([t],p)=\Phi^{t}(p)$,
we have the conditions below:
\begin{itemize}
    \item for each $m\in\bZ$, we have $\Phi^m=g^m\restriction_U$;
    \item the map $[t]\mapsto \Phi^{t}$ is a topological embedding of $\bR/\sqrt{2}\bZ$ into the group
    \[G\restriction_U:=\{h\restriction_U\mid h\in G\text{ and }h(U)=U\}\le\Homeo(U);\]
    \item for each $[t]\ne[0]$, we have $\fix \Phi^{t}\cap U=\varnothing$.
\end{itemize}
In this case, for $r\in \bR$ and $p\in U$, the map
\[
(U,g,r,p)\mapsto \Phi^r_{U,g}(p)\]
is definable.
\item If  $\models\on{flow}(U,g)\wedge \on{flow}(V,g)$, 
then for $p\in U\cap V$ and $r\in \bR$, we have
\[
\Phi^r_{U,g}(p)=\Phi^r_{V,g}(p).\]
    	\end{enumerate}
    \end{lem}
\begin{proof}
Applying Lemma~\ref{lem:unif-conv} for the definable set
\[
F_i:=\{h^{-1}g^s\mid s\in\bN\text{ and }\on{ang}(s,r)<1/i\}\sse G,\]
we immediately obtain a desired formula $\on{conv}(\gamma,\delta,u,\rho)$.

It is straightforward to check 
\[\on{flow}(u,\gamma)\equiv (\forall\rho \exists \delta)
[\conv(u,\gamma,\rho,\delta)\wedge (\rho\in\sqrt2\bZ\vee \fix\delta\cap u=\varnothing
)]\]
satisfies the desired conditions in (2). 
In particular, the uniquenss is a consequence of the fact that the formula $\on{conv}(U,g,r,h)$ uniquely determines the restriction of $h$ on $U$, as an approximation of the form 
\[\{g^{k_n}\restriction_U\}\]
satisfying \[k_n\longrightarrow r\] in $\bR/\sqrt{2}\bZ$. 
The definability of the flow in (2) and the independence on the choice of $U$ in part (3) also follow by the same reason, completing the proof.
    \end{proof} 

In the situation of Lemma~\ref{lem:flow}, we will say that $g$ defines a \emph{circular flow} on the open set $U$.
When we have $\conv(U,g,r,h)$, the element $g$ is viewed as an irrational rotation through a specified angle, and $h$ is the rotation of the $r$--multiple of this angle.
By the definability of $\Phi_{U,g}^r(p)$ for $p\in U$, we are justified to use an expression such as
\[
\Phi_{u,\gamma}^\rho(\pi)=\pi'\]
in an $L_{\agape}$ formula with the hypothesis that $\pi\in u$.
When the meaning is clear, we also use the more succinct notation
\[
g^r:=\Phi_{U,g}^r.\]
We are now ready to complete the proof of Theorem~\ref{thm:output-balls}: 

    \begin{proof}[Proof of Theorem~\ref{thm:output-balls}]
By Lemma~\ref{lem:flow}, we have an $L_{\agape}$ formula $\on{flows}_n(U,\underline{g},p)$ that expresses the following:
\begin{itemize}
    \item there exists some $\underline{V}=\{V_i\}$ such that 
    \[
    \models\cl V_i\sse\inte M\wedge g_i(V_i)=V_i\wedge
    \on{flow}(V_i,g_i)\] for each $i$, 
    and such that
    \[
    p\in U\in \cl U\sse \cap_i V_i;\]
    \item there exists a continuous bijection $[0,1]^n\longrightarrow \cl U$ defined by
    \[
    (r_1,\ldots,r_n)\mapsto \prod_i g_i^{r_i}(p);\]
    \item For all $r_i\in[0,1]$ and for all permutation $\sigma$ of $\{1,\ldots,n\}$, we have
    \[
    \prod_i g_i^{r_i}(p)=\prod_i g_{\sigma(i)}^{r_{\sigma(i)}}(p).\]
\end{itemize}
Here, it is implicitly required that
\[
\prod_{i=1}^{j-1} g_i^{r_i}(p)\in\cl U\]
for all $j\le n$, so that 
\[
\prod_{i=1}^{j} g_i^{r_i}(p)= g_{j}^{r_j}\circ\prod_{i=1}^{j-1} g_i^{r_i}(p)\]
is well-defined. The formula $\on{Param}_n$ is simply obtained 
from the map \[(U,g_i,r_i,p)\mapsto g_i^{r_i}(p).\] This proves part (\ref{p:flows1}).

For part (\ref{p:flows2}), we may identify $\cl U = Q^n(1)$ and $V=\inte Q^n(R)$ for some sufficiently large $R$.
We can then choose $n$ independent circular flows such that each flow rotates $U$ in some compact solid torus $B^{n-1}(1)\times S^1$ with the rotation number $1/\sqrt{2}$, and such that on the outside of $V$ the restrictions of the flows are the identity; see Figure~\ref{f:flow} (a), where a suitable homeomorphism is applied to $U$ for illustrative purposes. Such choices of flows will yield the desired conclusion.
\end{proof}

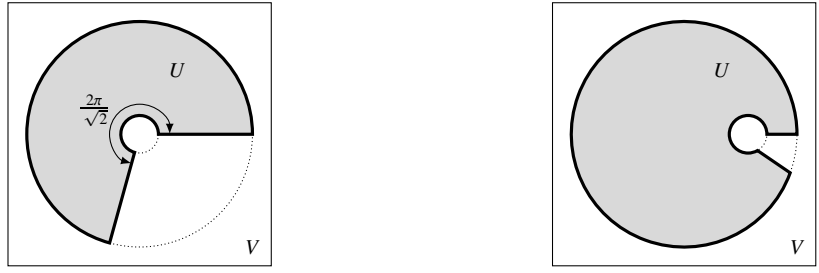
\begin{figure}[htb!]
\centering
\subcaptionbox{The ball $U$ is exactly $1/\sqrt2$ fraction of the domain of a flow.}
[.45\textwidth]{
\begin{tikzpicture}[scale=.5]
\draw[very thick,fill=gray!30] (0:.5) -- (0:3)   arc(0:254.558:3) -- (254.558:.5)  arc(254.558:0:.5) -- cycle;
\draw[densely dotted] (254.558:.5) arc(254.558:360:.5);
\draw[densely dotted] (254.558:3) arc(254.558:360:3);
\draw[latex-latex]  (0:.8) arc(0:254.558:.8) node[midway,left]{\tiny $\frac{2\pi}{\sqrt{2}}$};
\draw (1,1.7) node [] {\tiny $U$};
\draw (-3.5,-3.5) -- (-3.5,3.5)  -- (3.5,3.5)  -- (3.5,-3.5) -- cycle;
\draw (3,-3) node [] {\tiny $V$};
\end{tikzpicture}}
\qquad
\subcaptionbox{There may not be enough room for a desired measure--preserving flow.}[.45\textwidth]
{
\begin{tikzpicture}[scale=.5]
\draw[densely dotted] (1.7,0) circle (0.5);
\draw[densely dotted] (0,0) circle (3);
\draw[very thick,fill=gray!30] (2.2,0) arc(0:300:.5) -- (340:3) arc (340:0:3) -- cycle;
\draw (1,1.7) node [] {\tiny $U$};
\draw (-3.5,-3.5) -- (-3.5,3.5)  -- (3.5,3.5)  -- (3.5,-3.5) -- cycle;
\draw (3,-3) node [] {\tiny $V$};
\end{tikzpicture}}
\caption{The proof of Theorem~\ref{thm:output-balls} (2)
and a potential issue when $U$ is not ``spaciously collared''.}
\label{f:flow}
\end{figure}
We remark that in the measure preserving case, if $\vol(U)/\vol(V)$ is not sufficiently small, then there may not be enough room for a solid torus inside $V$ such that  $\cl U$ occupies $(1/\sqrt{2})$--fraction of the torus.
For instance, one may consider an annulus that is homeomorphic to $S^1\times I$, but which is equipped with a measure that is not the product of the Lebesgue measures on the two factors. Thus, the annulus may be ``throttled" in some interval as in Figure~\ref{f:flow} (b), and thus there may be no measure preserving flow that globally rotates the annulus.

\section{Parametrization of collar neighborhoods}\label{sec:collar}
Let us fix an integer $n>1$.
We now describe a definable parametrization of collar neighborhoods of the boundary of a compact $n$--manifold.
More specifically, we will establish the following.

\begin{thm}\label{thm:collar}
Then there exist formulae
\[
\on{collar}(\underline{\kappa}),\quad \on{collar-embed}(\underline{\kappa},\pi,\rho,\pi')\]
for some tuple $\underline{\kappa}$ of variables in the $\agape$ language
such that  each pair $(M,G)\in\cM_{(\vol)}$ with $\dim M=n$
satisfies the following:
\begin{enumerate}
    \item\label{p:col-suff} We have that  $\models(\exists{\underline{\kappa}})[\collar(\underline{\kappa})]$.
    \item\label{p:col-nece} Let $\underline{K}$ be a tuple of elements in $\agape(M,G)$  satisfying
    \[ \collar(\underline{K}).\]
    Then there exists a unique collar embedding
    \[u=u[\underline{K}]\co \partial M\times[0,1)\longrightarrow M\]
    of $\partial M$ such that for all points $p\in\partial M$ and $q\in M$, and for all $r\in[0,1)$ we have
    \[u(p,r)=q\Longleftrightarrow\left(\agape(M,G) \models\on{collar-embed}(\underline{K},p,r,q)\right).\]
    \end{enumerate}
\end{thm}

\subsection{Decomposition of a unit cube}\label{ss:brick}
Let us fix $n>1$. We will use a certain partition of a cube to parametrize a collar neighborhood of $\partial M$. 
We set
\begin{align*}
&\Lambda:=\{0,1\}^{n-1}\sse I^{n-1},\\ 
&\mathbf{0}:=(0,\ldots,0), \quad\mathbf{1}:=(1,\ldots,1)\in \Lambda,\\
&\mathbf{0}^k, \mathbf{1}^k\in \Lambda^k\sse (I^{n-1})^k&\text{ for }k>0,\\
&\on{len}(w):=k&\text{ for }w\in\Lambda^k,\\
&\pa(m):=m-2\floor{m/2}&\text{ for }m\in\omega.
\end{align*}
For convention, we also let
\[
\Lambda^0=\{\mathbf{0}^0\}=\{\varnothing\}.\]
By abuse of notation, we  move or remove parantheses rather freely and often write
\[
X^{(v_1,\ldots,v_k)}=X^{(v_1,\ldots,v_{k-1}),v_k}=X^{v_1,\ldots,v_k}\]
when the vector $(v_1,\ldots,v_k)$ is used to index certain objects $X^*$.
For each \[w=(v_1,\ldots,v_k)\in \Lambda^k\] with $k\in\omega$, we let 
$\bar S^w$ be the dyadic cube of side length $1/2^k$ that contains the following two points as opposite vertices:
\[
\sum_{i=1}^k {v_i}/{2^i}, \quad \sum_{i=1}^k{v_i}/{2^i}+ {\mathbf{1}}/{2^k}.\]
For instance, we have \[\bar S^\varnothing=I^{n-1},\quad \bar S^{\mathbf{0}}=[0,1/2]^{n-1},\quad
\bar S^{(\mathbf{0},\mathbf{1})}=[1/4,1/2]^{n-1},\] and so on.
We have partitions (with disjoint interiors):
\begin{align*}
I^{n-1} &= \bigcup\left\{ \bar S^w \middle\vert w\in \Lambda^k\right\}\qquad\qquad\qquad\text{ for each }
k\in\omega,\\
I^{n-1}\times[0,2)&=\bigcup\left\{ S^{w}:=\bar S^w\times
\left[
2-\frac1{2^{\on{len}(w)-1}},
2-\frac1{2^{\on{len}(w)}}
\right]
\ \middle\vert\  
w\in \bigcup_{k\in\omega}\Lambda^k\right\}.\end{align*}
We have a unique parametrization
\[\sigma^w\co I^n\longrightarrow S^w\]
of the regular cube $S^w$ obtained by a positive homothety and translation.



\subsection{The condition for a collar neighborhood}
Let us first consider the case that $(M,G)\in\cM$. For a tuple
\[
\underline{K}=(U_i,U_i^0,U_i^1,U_i^{0,v},U_i^{1,v},T_i^0,T_i^1,p_{i,0}^{\varnothing}, \on{hor}_i,\on{vert}_i,s_i^{0,v},s_i^{1,v}\mid 1\le i\le n\text{ and }v\in \Lambda)\]
in the universe of $\agape(M,G)$, we consider the collection of conditions with appropriate notation as itemized in (\ref{p:col-first}) through (\ref{p:col-last}) below; see Figure~\ref{f:collar} for an illustration when $n=2$.

 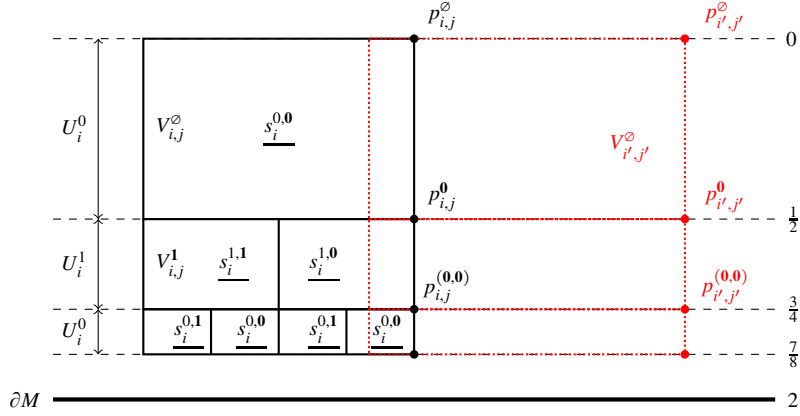
\begin{figure}\centering
\begin{tikzpicture}[scale=1.2]
\tikzstyle {bv}=[black,draw,shape=circle,fill=black,inner sep=1pt]
\tikzstyle {rv}=[red,draw,shape=circle,fill=red,inner sep=1pt]
\draw [ultra thick] (-4,0) node [left] {\tiny $\partial M$} -- (4,0) node [right] {\tiny $2$};
\draw [dashed] (-4,4) -- (4,4) node [right] {\tiny $0$};
\draw [dashed] (-4,2) -- (4,2) node [right] {\tiny $\frac12$};
\draw [dashed] (-4,1) -- (4,1) node [right] {\tiny $\frac34$};
\draw [dashed] (-4,.5) -- (4,.5) node [right] {\tiny $\frac78$};
\draw [thick] (-3,4) -- ++(3,0) -- ++(0,-2) -- ++(-3,0) -- cycle;
\draw (-2.7,3) node [] {\tiny $V_{i,j}^\varnothing$};
\draw (-2.7,1.5) node [] {\tiny $V_{i,j}^{\mathbf{1}}$};
\draw [thick] (-1.5,2) -- ++(0,-1.5);
\draw [thick] (-2.25,1) -- ++(0,-.5);
\draw [thick] (-.75,1) -- ++(0,-.5);
\draw [thick] (-3,2) -- ++(3,0) -- ++(0,-1) -- ++(-3,0) -- cycle;
\draw [thick] (-3,1) -- ++(3,0) -- ++(0,-.5) -- ++(-3,0) -- cycle;
\draw (.31,4.25) node [] {\tiny $p_{i,j}^{\varnothing}$};
\draw (.31,2.25) node [] {\tiny $p_{i,j}^{\mathbf{0}}$};
\draw (.37,1.25) node [] {\tiny $p_{i,j}^{(\mathbf{0},\mathbf{0})}$};
\draw [densely dotted,thick,red] (-.5,4) -- ++(3.5,0) -- ++(0,-2) -- ++(-3.5,0) -- cycle ++ (3,-2) node [] {\tiny $V_{i',j'}^\varnothing$};
\draw [densely dotted,thick,red] (-.5,2) -- ++(3.5,0) -- ++(0,-1) -- ++(-3.5,0) -- cycle;
\draw [densely dotted,thick,red] (-.5,1) -- ++(3.5,0) -- ++(0,-.5) -- ++(-3.5,0) -- cycle;
\draw (3.45,4.25) node [red] {\tiny $p_{i',j'}^{\varnothing}$};
\draw (3.45,2.25) node [red] {\tiny $p_{i',j'}^{\mathbf{0}}$};;
\draw (3.45,1.25) node [red] {\tiny $p_{i',j'}^{(\mathbf{0},\mathbf{0})}$};
\draw (3,4) node [rv] {};
\draw (3,2) node [rv] {};
\draw (3,1) node [rv] {};
\draw (3,.5) node [rv] {};
\draw (0,4) node [bv] {};
\draw (0,2) node [bv] {};
\draw (0,1) node [bv] {};
\draw (0,.5) node [bv] {};
\draw (-1.5,3) node [] {\tiny $\underline{s_i^{0,\mathbf{0}}}$};
\draw (-2,1.5) node [] {\tiny $\underline{s_i^{1,\mathbf{1}}}$};
\draw (-1,1.5) node [] {\tiny $\underline{s_i^{1,\mathbf{0}}}$};
\draw (-2.5,.75) node [] {\tiny $\underline{s_i^{0,\mathbf{1}}}$};
\draw (-1.8,.75) node [] {\tiny $\underline{s_i^{0,\mathbf{0}}}$};
\draw (-1,.75) node [] {\tiny $\underline{s_i^{0,\mathbf{1}}}$};
\draw (-.3,.75) node [] {\tiny $\underline{s_i^{0,\mathbf{0}}}$};
\draw[<->] (-3.5,2) -- (-3.5,4) node [midway,above,left] {\tiny $U_i^0$};
\draw[<->] (-3.5,1) -- (-3.5,2) node [midway,above,left] {\tiny $U_i^1$};
\draw[<->] (-3.5,.5) -- (-3.5,1) node [midway,above,left] {\tiny $U_i^0$};
\end{tikzpicture}
\caption{The condition $\on{COL}(M,G;\underline{K})$.}
\label{f:collar}
\end{figure}

\paragraph{\textbf{Condition $\on{COL}(M,G;\underline{K})$}}
\begin{enumerate}[(a)]
    \item\label{p:col-first}
    We have regular open sets $U^*$ and $U_1,\ldots,U_n$ such that
    \[\partial M\sse U^*=\bigcup_{1\le i\le n} U_i,\]
    and such that every regular open neighborhood of $\partial M$ contains $g(U^*)$ for some $g\in G$;
    moreover, each $U_i$ has finitely many components, and the closures of distinct components are disjoint.
    \item\label{p:col-ui}   
    We have dispersed (see Definition~\ref{d:dispersed}) regular open sets 
    \[
    U_i^0,U_i^1, U_i^{0,v},U_i^{1,v}\]
    for each $i\le n$ and $v\in \Lambda$;
    moreover, we have for each $\epsilon\in\{0,1\}$ that 
    \[
    U_i=U_i^0\oplus U_i^1,\quad
    U_i^\epsilon=\oplus_{v\in \Lambda}U_i^{\epsilon,v}.\]
    \item\label{p:col-pij} For each $i\le n$,
    we have $\on{hor}_i,\on{vert}_i\in G$ and 
    \[
    p_{i,0}^{\varnothing}
    \in T_i^0\sse T_i^1\in\PP^{\on{disc}}(\inte M)\]
    such that $T_i^0$ is a nonempty, finite, minimal $\on{hor}_i$--invariant set; moreover,
    the map
    \[(j,k)\mapsto p_{i,j}^{\mathbf{0}^k}:= \on{vert}_i^k \circ \on{hor}_i^j \left(p_{i,0}^{\varnothing}\right)\]
    is a bijection \[\{0,\ldots,\#T_i^0-1\}\times\omega\longrightarrow  T_i^1.\]
 
\item\label{p:col-uij} 
For each $i\le n$ and $j<\#T_i^0$, there exists a unique  $U_{i,j}\in\pi_0 U_i$ 
satisfying 
\[p_{i,j}^{\varnothing}\in\fr U_{i,j}.\]
For each $k\in\omega$, there also exists a unique $U_{i,j}^k\in\pi_0  U_i^{\pa(k)}$ such that
\[p_{i,j}^{\mathbf{0}^k}\in\fr U_{i,j}^k.\]
We further have closure--disjoint unions
\[
U_i=\bigsqcup_j U_{i,j},\qquad U_i^0=\bigsqcup_{j,k} U_{i,j}^{2k},\qquad U_i^1=\bigsqcup_{j,k} U_{i,j}^{2k+1}.\]
    \item\label{p:col-flows0}
    For each $i\le n$, we have $\underline{s}_i^{\varnothing}\in G^n$.
    Setting $V_{i,j}^\varnothing:= U_{i,j}^0$, we also have
    \begin{align*}
    &\models \on{flows}_n\left(V_{i,j}^{\varnothing},\underline{s}_i^{\varnothing},
    p_{i,j}^{\varnothing}\right),\\
    &\Psi_{i,j}^{\varnothing}:=\Psi\left[V_{i,j}^{\varnothing},\underline{s}_i^{\varnothing},
    p_{i,j}^{\varnothing}\right]\co I^n \longrightarrow\cl V_{i,j}^{\varnothing}.\end{align*}
    For all $k>0$ and $(v_1,\ldots,v_k)\in\Lambda^k$ we have that
    \[p_{i,j}^{(v_1,\ldots,v_k)}:=\on{vert}_i^k\circ \Psi_{i,j}^{\varnothing}\left(\sum_{i=1}^k \frac{v_i}{2^i},0\right)\in\fr U_{i,j}^{k-1}\cap\fr U_{i,j}^k.\]
    \item\label{p:col-vij}
    For each $(i,j,w=(v_1,\ldots,v_k))$ in the index set
     \[\II :=\left\{(i,j,w)    \middle\vert    1\le i\le n, 0\le j<\#T_i^0,  w\in \bigcup_{k\in\omega}\Lambda^k\right\},\] 
    there exists a unique \[V_{i,j}^w\in \pi_0 U_i^{\pa(k),v_k},\] the closure of which contains $p_{i,j}^w$.
    \item\label{p:col-flows1}
    For each $i\le n$ and $v\in\Lambda$, we have
     \[\underline{s}_i^{0,v},\underline{s}_i^{1,v}\in G^n.\]
    We further have that
    \begin{align*}
    &\models \on{flows}_n\left(V_{i,j}^{w},\underline{s}_i^{\pa(k),v_k},
    p_{i,j}^{w}\right),\\
    &\Psi_{i,j}^w:=\Psi\left[V_{i,j}^{w},\underline{s}_i^{\pa(k),v_k},
    p_{i,j}^{w}\right]\co I^n \longrightarrow\cl V_{i,j}^{w}.
    \end{align*}
 
   \item\label{p:col-equiv} 
   For each $(i,j,\varnothing)\in\II$, there exists a homeomorphism \[\Psi_{i,j}\co I^{n-1}\times[0,2]\longrightarrow \cl U_{i,j}\] such that for each $w\in\Lambda^k$ we have \[\Psi_{i,j}\restriction_{S^w}=\Psi_{i,j}^w\circ (\sigma^w)^{-1},\] and such that \[\Psi_{i,j}(I^{n-1}\times \{2\})\sse \partial M.\]

    \item\label{p:col-last}
     If $x\in \cl U_{i,j}\cap \cl U_{i',j'}$, then some $v,v'\in I^{n-1}$ and $t\in[0,2]$ satisfy    \[    x = \Psi_{i,j}(v,t)=\Psi_{i',j'}(v',t).\]  
     Moreover, in this case we require that for each $t'\in[0,2]$, we have \[ \Psi_{i,j}(v,t')=\Psi_{i',j'}(v',t').\]
 \end{enumerate}

We now make three claims. First, these conditions are first order expressible. 
Second, these conditions produce a definable collar embedding; for this, we will actualy need only the conditions (\ref{p:col-equiv}) and (\ref{p:col-last}).
Third, every pair $(M,G)\in\cM$ satisfies these conditions with a suitable choice of $\underline{K}$. 

The first point is trivial to check from the preceding results, possibly except for the continuity condition in (\ref{p:col-equiv}) at the level--$2$ subset of $I^{n-1}\times[0,2]$. At such a point $x_0$, we then can simply require the convergence of the values of the form
\[\Psi_{i,j}^w\circ(\sigma^w)^{-1}(x)\]
whenever $x\in S^w$ gets arbitrarily close to $x_0$; we also require the bijectivity of the resulting map onto $\cl U_{i,j}$.
We can now let $\on{collar}(\underline{\kappa})$ be the formula expressing the condition $\on{COL}(M,G;\underline{K})$. 

Regarding the second point, we note the following:
\begin{claim*}\label{cla:col-emb}
Under the hypothesis $\on{COL}(M,G;\underline{K})$, we have a collar embedding
    \[u=u\left[\underline{K}\right]\co \partial M\times[0,2]\longrightarrow M\]
which is unambiguously defined by
    \[u(\Psi_{i,j}(v,2),r)=\Psi_{i,j}(v,r)\]
for all \[(i,j,\varnothing)\in\II,\quad v\in I^{n-1},\quad r\in[0,2].\]
In particular, the image of the level--$2$ set under the map $u$ coincides with $\partial M$.
\end{claim*}
\begin{proof}
The well-definedness and the injectivity follow from the condition~(\ref{p:col-last}) above. This map $u$ is continuous because $\Psi_{i,j}$ is for all $i$ and $j$. The condition~(\ref{p:col-equiv}) further implies that this map $u$ is a collar embedding of the boundary.
\end{proof}

From the above claim and from the definability of $\Psi_{i,j}^w$, we obtain the desired formula $\on{collar-emb}(\underline{\kappa}, \pi,\rho,\pi')$ expressing the map $u$.  We complete the proof of part (\ref{p:col-nece}) in Theorem~\ref{thm:collar} by simply reparametrizing $u$ so that the level--$0$ set corresponds to the boundary.

For the third claim, and hence part (\ref{p:col-suff}) of the theorem, we note that the condition~(\ref{p:col-first}) is equivalent to  $\cl U^*$ being contained in a collar neighborhood.
Hence, we may simply start with a homeomorphism 
\[
u\co \partial M\times[0,2]\longrightarrow \cl U^*\]
that satisfies \[u(x,2) = x\in\partial M.\]
Using Ostrand's theorem (Lemma~\ref{lem:dim-facts} (\ref{p:ostrand})), 
we can write
\[
\partial M=\bigcup_{1\le i\le n} \cl \bar W_i=\bigcup_{1\le i\le n} \bar W_i\]
for some $\cl \bar W_i\sse\partial M$, each of whose components $\cl W_{i,j}$ is homeomorphic to $I^{n-1}$.
We have a natural homeomorphism
\[
u_{i,j}\co I^{n-1}\times[0,2]\longrightarrow U_{i,j}:=u(W_{i,j}\times [0,2]).\]
Denote by $p_{i,j}^{\mathbf{0}^k}$  the image of $(0,0,\ldots,0,2-1/2^{k-1})$ under this homeomorphism.
We can find a homeomorphism $\on{hor}_i$ that permutes the components $\cl U_{i,j}$ of $\cl U_i$ as in condition~(\ref{p:col-uij}).
We let $T_i^0:=\{p_{i,j}^\varnothing\}_j$ and $T_i^1:=\{p_{i,j}^{\mathbf{0}^k}\}_{j,k}$.
We further define
\[
V_{i,j}^w:=u_{i,j}(S^w),\]
and set 
\[
U_{i,j}^k:=\oplus_{w\in\Lambda^k} V_{i,j}^w, \quad U_{i}^0:=\bigsqcup_{j,k} U_{i,j}^{2k}.\]
The regular open sets $U_i^1, U_i^{0,v},U_i^{1,v}$ are similar and straightforward to define.
The homeomorphism $\on{vert}_i$ is clearly defined, so that $\on{vert}_i(p_{i,j}^w)=p_{i,j}^{w,\mathbf{0}}$. 
After decomposing $U_{i,j}^k$ modeled on $\{S^w\}$, we find $\underline{s}_i^{\pa(k),v}$ for the current setup using the uniform convergence theorem. Here, it is crucial that the diameters of the cubes $V_{i,j}^w$ converge to zero as they approach the boundary. This completes the proof of the case $(M,G)\in\cM$.

Slightly more care is needed in the measure preserving case $(M,G)\in\cM_{\vol}$. To guarantee the existence of a measure preserving flow avoiding issues as described in Figure~\ref{f:flow}, we need that the components of the supports of flow-generating homeomorphsms $\underline{s}_i^v$ to be sufficiently far from each other. More precisely, we will pick a sufficiently large $n_0>0$ depending on $M$, and replace
 condition~(\ref{p:col-flows1}) by the following two conditions; we also change the definition of the tuple $\underline{\kappa}$, which is now required to contain the group tuple variables $\underline{s}_{i,j}^w$ as below.
 \begin{enumerate}[(a)]
 \item[(f)']
    For each $k\in\{1,\ldots,n_0\}$, $\epsilon\in\{0,1\}$ and $w=(v_1,\ldots,v_k)\in\Lambda^k$, we have
    \[\underline{s}^{w}, \underline{s}^{\epsilon,w} \in G^n.\]
    We further have that
    \begin{align*}
    &\models \on{flows}_n\left(V_{i,j}^{w},\underline{s}_i^{w},
    p_{i,j}^{w}\right),\\
    &\Psi_{i,j}^w:=\Psi\left[V_{i,j}^{w},\underline{s}_i^{w},
    p_{i,j}^{w}\right]\co I^n \longrightarrow\cl V_{i,j}^{w}.
    \end{align*}
 \item[(f)'']
    For each $k>n_0$ and $w=(v_1,\ldots,v_k)$,
    after setting $w':=(v_{k-n_0+1},\ldots,v_k)$ we have that
    \begin{align*}
    &\models \on{flows}_n\left(V_{i,j}^{w},\underline{s}_i^{\pa(k),w'},
    p_{i,j}^{w}\right),\\
    &\Psi_{i,j}^w:=\Psi\left[V_{i,j}^{w},\underline{s}_i^{\pa(k),w'},
    p_{i,j}^{w}\right]\co I^n \longrightarrow\cl V_{i,j}^{w}.
    \end{align*}
 \end{enumerate}

Part (\ref{p:col-nece}) of Theorem \ref{thm:collar} is still proved in the same way, even independently of the choice of $n_0\ge 1$. For part (\ref{p:col-suff}), we choose $n_0$ sufficiently large, under a fixed metric and a measure on some chart neighborhood of $M$.
We will require that for each fixed $w':=(v_1,\ldots,v_{n_0})\in\Lambda^{n_0}$, 
each open set in the collection
\[
\left\{V_{i,j}^w\middle\vert w=(\ldots,v_1,\ldots,v_{n_0})\in\bigcup_{k>{n_0}}\Lambda^k\right\}\]
is contained in some closure--disjoint collection of open balls
\[
\left\{W_{i,j}^w\middle\vert w=(\ldots,v_1,\ldots,v_{n_0})\in\bigcup_{k>{n_0}}\Lambda^k\right\}\]
 with the additional requirement that $\vol(V_{i,j}^w)/\vol(W_{i,j}^w)$ is sufficiently small, in the sense of Theorem~\ref{thm:output-balls}. This guarantees the existence and the convergence of each measure preserving homeomorphism of the required form $\underline{s}^{\epsilon,w'}$, thus completing the proof.

\section{Completing the proof}\label{sec:homeo}
Cheeger and Kister~\cite{CK1970} proved that there exist only countably many homeomorphism types of compact manifolds. A key step in their proof is that the topological type of a manifold is invariant under ``small'' perturbations, in some quantitatively precise sense. 
As is more concretely described below, 
this step will be crucial for the construction of the sentences $\phi_M$ and $\phi_M^{\vol}$.
    
For positive integers $n,k$ and $\ell$, we denote by $\EE(n,k,\ell)$ the set of all tuples of embeddings
\[
\underline f=(f_{1,1},\ldots,f_{1,k},f_{2,1},\ldots,f_{2,\ell})\]
from $Q^n(2)$ to $\bR^{2n+1}$ such that the following conditions hold:
\begin{enumerate}[(i)]
    \item The following set is a compact connected $n$--manifold:
\[M=C(\underline f):=\bigcup_{i,j} \im f_{i,j}\sse\bR^{2n+1}.\]
    \item There exists a collar $u\co \partial M\times[-2,2]\longrightarrow M$ such that
    $u(x,-2)=x$ for all $x$.
    \item We have that
    \[M\setminus u(\partial M\times[-2,0))
    \sse\bigcup_i f_{1,j}(\inte Q^n(1))
    \sse\bigcup_i f_{1,j}(Q^n(2))
    \sse M\setminus u(\partial M\times[-2,-1]).\]
    \item
    For each $i=1,\ldots,\ell$, 
    the restriction
    \[
    f_{2,j}\restriction_{Q^{n-1}(2)\times\{-2\}}\]
    is an embedding of $Q^{n-1}(2)$ into $\partial M$
    such that
    \[
    f_{2,j}(x,t)=u(f_{2,j}(x,-2),t),\]
    where here $x\in Q^{n-1}(2)$ and $t\in[-2,2]$, and such that
    \[
    \partial M=\bigcup_i f_{2,j}(\inte Q^{n-1}(1)\times\{-2\}).\]
\end{enumerate}

Every compact $n$--manifold $M$ is homeomorphic to $C(\underline{f})$
for some tuple \[\underline{f}=(f_{i,j})\in\EE(n,k,\ell)\] as above, which we call as a \emph{parametrized cover} of $C(\underline{f})$.
The space $\EE(n,k,\ell)$ inherits the uniform separable metric from the space \[C^0(Q^n(2),\bR^{(2n+1)(k+\ell)}).\] The proof of Cheeger and Kister essentially boils down to the following rigidity result,
along with a deep result of Edwards and Kirby
on
deformation of embeddings in manifolds~\cite{EK1971}.
\begin{lem}\cite{CK1970}\label{lem:ck}
For each $\underline{f}\in\EE(n,k,\ell)$ and for each $\epsilon>0$,  there exists $\delta>0$ such that every $\underline{ g}\in\EE(n,k,\ell)$ that is at most $\delta$--far from $\underline f$ admits a homeomorphism 
\[ C(\underline{f})\longrightarrow  C(\underline{g})\] 
that is at most $\epsilon$--far from the identity map. \end{lem}

We choose a sufficiently small $\delta>0$ for which the conclusion of Lemma~\ref{lem:ck} holds, and call it as a \emph{Cheeger--Kister number} of $\underline f\in\EE(n,k)$; for our purposes, we will further require $\delta$ to be rational.
Our strategy for proving Theorem~\ref{thm:main} is providing a sentence in $L_{\agape}$ which is modeled by an input manifold $M\sse\bR^{2n+1}$, such that the sentence holds for the structure $\agape(N,H)$
if and only if $N$ admits an embedding into Euclidean space that is within the Cheeger--Kister number of a fixed parametrized cover of $M$.

In order to execute this strategy, let us fix a pair $(M,G)\in\cM_{(\vol)}$ with $\dim M=n$.
We will slightly modify the definition in Theorem~\ref{thm:output-balls} by affine transformations, so that $\Psi[U,\underline{g},p]$ is a map from
$Q^n(2)$ into $M$, sending $(-2,\ldots,-2)$ to $p$.

We let $k$ and $\ell$ be positive integers, and consider a tuple
\[
\underline{f}=(f_{1,1},\ldots,f_{1,k},f_{2,1},\ldots,f_{2,\ell})\]
of functions in $C^0(\bR^n,\bR^{2n+1})$. 
Let us denote by $\on{EMB}(M,G;\underline{f})$ the collection of all the conditions below
from (\ref{p:emb-f}) through (\ref{p:emb-compat}); see also Figure~\ref{f:par-cover}:
\be[(a)]
\item\label{p:emb-f} each $f_{i,j}$ restricts to an embedding of $Q^n(2)$ into $\bR^{2n+1}$;
\item\label{p:emb-ugp} for all indices as above, we have some
\[U_{i,j}\in \ro(M),\quad  p_{i,j}\in\inte M,\quad  \underline{g}_{i,j}\in G^n\]
satisfying $\on{flows}_n(U_{i,j},\underline{g}_{i,j},p_{i,j})$, corresponding to the homeomorphism
\[ h_{i,j}:=\Psi\left[U_{i,j},\underline{g}_{i,j},p_{i,j}\right]\co Q^n(2)\longrightarrow \cl U_{i,j}\sse\inte M;\]
\item\label{p:emb-collar} 
there exists a collar \[u\co \partial M\times[-3,2]\longrightarrow M\]
such that $u\restriction_{\partial M\times\{-3\}}=1\restriction_{\partial M}$,
and such that
\[M\setminus u(\partial M\times[-3,0))
    \sse\bigcup_j h_{1,j}(\inte Q^n(1))
    \sse\bigcup_j h_{1,j}(Q^n(2))
    \sse M\setminus u(\partial M\times[-3,-1]);\]
\item\label{p:emb-boundary}    
for each $j=1,\ldots,\ell$, 
    the restriction
    \[
    h_{2,j}\restriction_{Q^{n-1}(2)\times\{-2\}}\]
    is an embedding of $Q^{n-1}(2)$ into $u(\partial M\times\{-2\})$
    such that
    \[
    h_{2,j}(x,t)=u(h_{2,j}(x,-2),t),\]
    where here $x\in Q^{n-1}(2)$ and $t\in[-2,2]$, and such that
    \[
    u(\partial M\times\{-2\})=\bigcup_j h_{2,j}(\inte Q^{n-1}(1)\times\{-2\}).\]
\item\label{p:emb-compat} whenever $x\in \cl U_{a,b}\cap \cl U_{c,d}$ for some $a,b,c,d$, we have \[f_{a,b}\circ h_{a,b}^{-1}(x)=f_{c,d}\circ h_{c,d}^{-1}(x).\]
\ee
The condition $\on{EMB}(M,G;\underline{f})$ implies that \[\underline{f}\circ \underline{h}^{-1}:=(f_{i,j}\circ h_{i,j}^{-1})\] defines an embedding  \[M':=M\setminus u(\partial M\times[-3,-2))\hookrightarrow\bR^{2n+1},\] and that the tuple $\underline{f}$ is a parametrized cover of the image.

\begin{figure}\centering
\begin{tikzpicture}[scale=.7]
\draw [ultra thick] (-4,-1.5) node [left] {\tiny $\partial M$} -- (4,-1.5) node [right] {\tiny $-3$};
\draw[->] (-3.5,-1) -- (-3.5,.5) node [above,left] {\tiny $M'$};
\draw [dashed] (-4,-1) -- (4,-1) node [right] {\tiny $-2$};
\draw [dashed] (-4,1) -- (4,1) node [right] {\tiny $\phantom{-}2$};
\draw [dashed] (-4,-.5) -- (4,-.5) node [right] {\tiny $-1$};
\draw [dashed] (-4,0) -- (4,0) node [right] {\tiny $\phantom{-}0$};
\draw [thick,rounded corners=5] (-3,0.25) -- ++(1,-.5) -- ++(2.5,.5) -- ++(0,.5) -- ++(-3.1,0.5) -- cycle ++ (0.7,0.3) node [] {\tiny $\im h_{1,i}$};
\draw [thick] (-0,-1) -- ++(3,0) -- ++(0,2) -- ++(-3,0) -- cycle ++ (1.7,1.2) node [] {\tiny $\im h_{2,j}$};
\end{tikzpicture}
\caption{Parts (c) and (d) of the condition $\on{EMB}(M,G;f)$.}
\label{f:par-cover}
\end{figure}
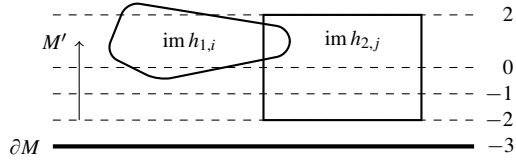

Recall the domain of the sort symbol $\cont_{n,2n+1}$ is $C^0(\bR^n,\bR^{2n+1})$.
By the preceding results, there exists a formula 
\[\on{Embed}_{n,k,\ell}(\underline{\chi}),\] expressing $\on{EMB}(M,G;\underline{f})$ in $\agape(M,G)$. 
We emphasize that although the maps $h_{i,j}$ do not belong to the universe of $\agape(M,G)$,  Theorem~\ref{thm:collar}
together with our access to the real numbers enables us to use such expressions. Let us record this fact:

\begin{lem}\label{lem:emb}
Let $(M,G)\in\cM_{(\vol)}$ satisfy $\dim M=n$.
For positive integers $n,k$ and $\ell$, 
there exists a formula $\embed_{n,k,\ell}(\underline{\chi})$ with a $(k+\ell)$--tuple of $\cont_{n,2n+1}$ variables 
\[
\underline{\chi}=(\chi_{1,1},\ldots,\chi_{1,k},\chi_{2,1},\ldots,\chi_{2,\ell})\]
in the $\agape$ language such that
\[\models\embed_{n,k,\ell}(\underline{f})\]
if and only if the condition $\on{EMB}(M,G;\underline{f})$ is satisfied.
\end{lem}

We can now establish the main result of this paper.
\begin{proof}[Proof of Theorem~\ref{thm:main}]
We may assume that $n:=\dim M>1$ and that $M\sse\bR^{2n+1}$.
Consider a parametrized cover
\[\underline{f}\in\EE(n,k,\ell)\sse C^0(Q^n(2),\bR^{(2n+1)(k+\ell)})\] of $M=C(\underline{f})$.
We have a corresponding Cheeger--Kister rational number \[\delta=\delta(M,\underline{f})>0.\]
Let us pick  $\delta_0>0$ such that \[\sup_{\|x-y\|\le \delta_0}\| \underline{f}(x)-\underline{f}(y)\|<\delta/3.\]
We can find a partition $\{C_1,\ldots,C_s\}$ of $Q^n(2)$ having diameters less than $\delta_0$ such that each $C_i$ is the intersection of $Q^n(2)$ with a cube with rational corners.
Each $C_i$ is definable in $L_{\agape}$, since so is every rational number.
We arbitrarily pick $x_i$ in $C_i$, and choose $q_i\in\bQ^{(2n+1)(k+\ell)}$ such that
    \[
    \|\underline{f}(x_i)-q_i\|<\delta/3.\]    
    
Let us now consider the following conditions
for an arbitrary $(N,H)\in\cM_{(\vol)}$, which are first order expressible in $L_{\agape}$ by preceding results:
\begin{itemize}
\item $\dim N=n$;
\item some tuple $\underline{g}\in C^0(\bR^n,\bR^{(2n+1)(k+\ell)})$ satisfies that 
\[\agape(N,H)\models \embed_n(\underline{g}),\]
and also 
\[\sup_{x\in C_i}\| \underline{g}(x)-q_i\|<\delta/3.\]
\end{itemize}
The above conditions are obviously met in the case when $(N,H)=(M,G)$.
We also note that for each $x\in C_i$ that
    \[
    \|  \underline{g}(x) - \underline{f}(x)\|
    \le
    \|  \underline{g}(x) - q_i\| + \|q_i - \underline{f}(x_i)\| +  \|\underline{f}(x_i) - \underline{f}(x)\|<\delta.\]
By Lemma~\ref{lem:ck}, we see that $N$ is homeomorphic to $M$.\end{proof}

\section{Further questions}\label{sec:questions}
A large number of interesting open questions remain. We already mentioned Question~\ref{que:homeo-action}.
Part of the motivation for this question is the theory of critical regularity of groups, which seeks to distinguish between diffeomorphism groups of various regularities of a given manifold by the isomorphism types of finitely generated subgroups; cf.~\cite{KK2020crit,Mann:aa}. 
Along this line of question, one may ask whether or not the $C^k$--analogue of Theorem~\ref{thm:main} holds.

    \begin{que}\label{que:first-order-diffeo}
    Let $M$ be a compact, connected, smooth manifold, and let $N$ be an arbitrary smooth manifold. Is there a sentence $\phi_{k,M}$
    in the language of groups such that if $\Diff^k(N)$ satisfies $\phi_{k,M}$ then $N$ is diffeomorphic to $M$?
    \end{que}
    
    Relatedly, leaving the framework of first order rigidity, we have the following.
    
\begin{que}\label{que:critical-reg}
    Let $M$ be a compact, connected, smooth manifold. Is there a finitely generated (or countable) group $G_M$ such that $G_M$
    acts faithfully by $C^k$ diffeomorphisms of a compact, connected, smooth manifold $N$ of the same dimension as $M$ if and only $N$ is $C^k$ diffeomorphic to $M$?
    \end{que}
    
    The discussion in the present article depended heavily on the compactness of the comparison manifold.

    \begin{que}\label{que:noncompact}
    Let $M$ be an arbitrary manifold. Under what conditions is there a sentence $\phi_M$ in the language of groups
    such that if $N$ is an arbitrary manifold then
    $\Homeo(N)$ satisfies $\phi_M$ if and only if $N$ is homeomorphic to $M$? More generally, under what conditions does $\Homeo(M)\equiv\Homeo(N)$ imply $M\cong N$?
    \end{que}
    
    We conclude by asking what the weakest hypotheses on $\cH$ can be.
    
    \begin{que}\label{que:minimal assumptions}
    For what classes of subgroups of $\Homeo(M)$ do the conclusions of Theorem~\ref{thm:main}
    hold?
    \end{que}

    A partial answer to Question~\ref{que:minimal assumptions} is given in~\cite{KdlN2024}, as noted
    in Remark~\ref{rem:other-results}.

\section*{Acknowledgements}
The first and the third author are supported by Mid-Career Researcher Program (RS-2023-00278510) through the National Research Foundation funded by the government of Korea.
The first and the third authors are also supported by KIAS Individual Grants (MG073601 and MG084001, respectively) at Korea Institute for Advanced Study
and by Samsung Science and Technology Foundation under Project Number SSTF-BA1301-51. The second author is partially supported by NSF Grants DMS-2002596
and DMS-2349814, Simons Foundation International Grant SFI-MPS-SFM-00005890. The authors thank M.~Brin, J.~Hanson and O.~Kharlampovich, and for helpful discussions. The authors are deeply grateful to C.~Rosendal for introducing the result of Cheeger--Kister to them.
  
  \bibliographystyle{amsplain}
  \bibliography{ref}
  
\end{document}